\documentclass{amsart}
\usepackage{amsmath,amsthm}
\usepackage{amsfonts,amssymb}

\usepackage{graphicx}
\usepackage{sidecap}
\usepackage{wrapfig}
\usepackage{ifpdf}

\newtheorem{thm}{Theorem}[section]

\newtheorem{prop}[thm]{Proposition}

\newtheorem{defn}[thm]{Definition}

\theoremstyle{remark}

\makeatletter
\newcommand{\bddots}{%
  \mathinner{\mkern1mu\raise\p@\vbox{\kern7\p@\hbox{.}}\mkern2mu
    \raise4\p@\hbox{.}\mkern2mu\raise7\p@\hbox{.}\mkern1mu}}
\makeatother

\def\a{{\alpha}}

\def\ab{{\mathbf a}}
\def\ib{{\mathbf i}}

\def\kb{{\mathbf k}}

\def\jb{{\mathbf j}}
\def\kb{{\mathbf k}}
\def\lb{{\mathbf l}}

\def\sb{{\mathbf s}}
\def\tb{{\mathbf t}}

 \def\NN{{\mathbb N}}
 \def\CA{{\mathcal A}}
 
 \def\CP{{\mathcal P}}
 \def\CK{{\mathcal K}}
 \def\CI{{\mathcal I}}
 \def\CL{{\mathcal L}}
 \def\CT{{\mathcal T}}
 \def\CTC{{\mathcal T\!C}}
 \def\CTS{{\mathcal T\!S}}
 
 \def\CH{{\mathcal H}}
 
 \def\HH{{\mathbb H}}
 \def\KK{{\mathbb K}}

 \def\RR{{\mathbb R}}
 \def\ZZ{{\mathbb Z}}
 \def\A{{\mathcal A}}
 
 \def\S{{\mathcal S}}
\newcommand{\e}{\mathrm{e}}
\newcommand{\ve}{\mathrm{v}}
\newcommand{\tr}{{\mathsf {tr}}}
\newcommand{\TC}{{\mathsf {TC}}}
\newcommand{\TS}{{\mathsf {TS}}}

 \def\sign{\operatorname{sign}}
 
 \def\sspan{\operatorname{span}}
\def \la {\langle}
\def \ra {\rangle}

\newcommand{\wh}{\widehat}

 \ifpdf
 \DeclareGraphicsExtensions{.pdf,.png,.jpg}
 \else
 \DeclareGraphicsExtensions{.eps}
 \fi
 \graphicspath{{images/}}

\begin{document}

\title[Discrete Fourier analysis with lattice]
{Discrete Fourier analysis with lattices on planar domains}

\author{Huiyuan Li}
\address{Institute of Software\\
Chinese Academy of Sciences\\ Beijing 100190, China}
\email{hynli@mail.rdcps.ac.cn}
\author{Jiachang Sun}
\address{Institute of Software\\
Chinese Academy of Sciences\\ Beijing 100190, China}
\email{sun@mail.rdcps.ac.cn}
\author{Yuan Xu}
\address{  Department of Mathematics\\ University of Oregon\\
    Eugene, Oregon 97403-1222.}
\email{yuan@math.uoregon.edu}

\begin{abstract}
A discrete Fourier analysis associated with translation lattices is developed
recently by the authors. It permits two lattices, one determining the integral
domain and the other determining the family of exponential functions. Possible
choices of lattices are discussed in the case of lattices that tile $\RR^2$ and
several new results on cubature and interpolation by trigonometric, as well
as algebraic, polynomials are obtained.
\end{abstract}

\date{\today}
\thanks{The first  authors was supported by NSFC Grants
10601056 and 10971212. The second author was supported by NSFC Grant 60970089 and  National Basic
Research Program Grant 2005CB321702.  }

\maketitle

\section{Introduction}
\setcounter{equation}{0}

A framework of discrete Fourier analysis associated with translation
tiling was developed recently in \cite{LSX1}, based on the principle that
if $\Omega$ is a bounded open set that tiles $\RR^d$ with the lattice
$L= A\ZZ^d$,  then the family of exponentials $\{\e^{ 2\pi i \alpha \cdot x} :
\alpha \in L^\perp\}$, where $L^\perp = A^{-\tr}\ZZ^d$ is the dual lattice
of $L$, forms an orthonormal basis in $L^2(\Omega)$ (\cite{F}). Our
set up permits two lattices, one determining the integral domain and the
other determining the exponentials that are orthogonal under the discrete
inner product. The case that both lattices have the regular hexagon as
fundamental domain was studied in \cite{LSX1} to illustrate the main set up,
which leads to new cubature formula and Lagrange interpolation for
trigonometric polynomials on hexagonal domains and equilateral triangles,
as well as results for algebraic polynomials on the region bounded by
Steiner's hypocycloid. This is extended to three dimension in \cite{LX1},
giving results on cubature and interpolation on the rhombic dodecahedron
and the tetrahedron, and further extended to $\RR^d$ in \cite{LX2} for
$\CA_d$ type lattice. In \cite{LSX2}, the two lattices are chosen differently
with fundamental domains being a square and a rhombic (rotation of the
square by $90^\circ$), respectively. The choice leads to, surprisingly, one
family of minimal cubature for product Chebyshev weight on $[-1,1]^2$, first
discovered by working with common zeros of orthogonal polynomials of
two variables.  An extension to three dimension gives a family of cubature
formulas on the cube that have the smallest number of nodes among all
known formulas, which coincides, rather surprisingly, with the cubature
discovered in \cite{MVX} by a totally different method.

The two lattices in \cite{LSX2} were chosen for the purpose of obtaining
algebraic cubature formulas on the square. Its success prompts us to ask
what other choices are possible. In the present paper, we try to answer this
question in the case of $\RR^2$. Up to affine transforms, there are essentially
two types of translation tiling in $\RR^2$ with fundamental domain being
either squares or regular hexagons. Their combinations in our discrete
Fourier analysis, however, yield several distinct cases, including several cases
not covered in our previous studies. One that is of particular interesting has
one tiling sets as the regular hexagon and the other as the rotation of the
regular hexagon by $90^\circ$ (see Section 3.5), which leads to another
set of cubature and interpolation on the equilateral triangle, different from
those obtained in \cite{LSX1}. In order to present the main idea without
being overwhelmed by notations and numerous formulas,  we shall work mostly
with cubature formulas, a central part but by no means all of discrete Fourier
analysis, unless other results are deemed novel enough to warrant
inclusion.

The paper is organized as follows. In the following section we recall the
framework developed in \cite{LSX1} and use it to treat the classical product
discrete Fourier analysis on the plane, which illustrates well what can be
expected in the non-classical settings. Section 3 is divided into a number
of subsections, each deals with one distinct choice of two lattices.

\section{Discrete Fourier analysis with lattice}
\setcounter{equation}{0}

In the first subsection, we give a succinct recount of the framework of discrete
Fourier analysis with tiling in \cite{LSX1}. We refer to \cite{CS} for lattices, tiling
and various related topics, and refer to \cite{DM, Ma} for some applications
of discrete Fourier analysis in several variables. In the second subsection, we
illustrate the general theory by using it to recover the classical product discrete
Fourier analysis on the square.

\subsection{Discrete Fourier analysis}
A lattice $L$ in $\RR^d$ is a discrete subgroup $L =  L_A := A\ZZ^d$, where
$A$, called a generator matrix, is nonsingular. A bounded set $\Omega$
of $\RR^d$, called the fundamental domain of $L$, is said to tile
$\RR^d$ with the lattice $L$ if
$$
  \sum_{\alpha \in L} \chi_\Omega (x + \alpha) = 1, \qquad
      \hbox{for almost all $x \in \RR^d$},
$$
where $\chi_\Omega$ denotes the characteristic function of $\Omega$. We
write this as $\Omega + L = \RR^d$.  For a given lattice $L_A$, the
dual lattice $L_A^\perp$ is given by $L_A^\perp =  A^{- \tr}\ZZ^d$.
According to a result of Fuglede \cite{F}, a bounded open set $\Omega$ tiles
$\RR^d$ with the lattice $L$ if, and only, $\{\e^{2 \pi i \alpha \cdot x}: \alpha
\in L^\perp\}$ is an orthonormal basis with respect to the inner product
\begin{equation} \label{ipOmega}
  \langle f, g \rangle_\Omega = \frac{1}{\mathrm{mes}(\Omega)}
      \int_\Omega f(x) \overline{g(x)} dx.
\end{equation}
For $L = A\ZZ^d$, the measure of $\Omega$ is equal to $|\det (A)|$. Since
$L^\perp_A = A^{-\tr}\ZZ^d$, we can write $\alpha = A^{-\tr} k$ for $\alpha \in
L_A^\perp$ and $k \in \ZZ^d$, so that $\e^{2\pi i \alpha \cdot x}
= \e^{2 \pi i k^{\tr} A^{-1} x}$.

For our discrete Fourier analysis, the boundary of $\Omega$ matters. We shall
fix an $\Omega$ such that $0 \in \Omega$ and  $\Omega + A\ZZ^d =  \RR^d$
holds {\it pointwisely} and {\it without overlapping}.

\begin{defn} \label{def:N}
Let  $\Omega_A$ and $\Omega_B$ be the fundamental domains of $A\ZZ^d$
and $B\ZZ^d$, respectively. Assume {\it all entries of the matrix
$N:=B^\tr A$ are integers}.
Define
$$
\Lambda_N :=   \{ k \in \ZZ^d: B^{-\tr}k \in \Omega_A\} \quad \hbox{and} \quad
\Lambda_{N}^\dag:=   \{ k \in \ZZ^d: A^{-\tr}k \in \Omega_B\}.
$$
Furthermore, define the finite dimensional subspace of exponential functions
$$
\CH_N := \sspan \left\{\e^{2\pi i\, k^\tr A^{-1} x}, \, k \in \Lambda_{N}^\dag \right \}.
$$
\end{defn}

The main result in the discrete Fourier analysis is the following theorem:

\begin{thm} \label{thm:df}
Let $A, B$ and $N$ be as in Definition \ref{def:N}.  Define
$$
    \langle f, g \rangle_N =  \frac{1}{|\det (N)|}
        \sum_{j \in \Lambda_N } f(B^{-\tr} j ) \overline{g(B^{-\tr} j )}
$$
for $f, g$ in $C(\Omega_A)$, the space of continuous functions on $\Omega_A$.
Then
\begin{equation}\label{c-d-inner}
    \langle f, g \rangle_{\Omega_A} =  \langle f, g \rangle_{N}, \qquad f, g \in \CH_N.
\end{equation}
\end{thm}

It follows readily that \eqref{c-d-inner} gives a cubature formula exact for
functions in $\CH_N$. Furthermore, it also implies a Lagrange interpolation
by exponential functions. Let $\CI_N f$ denote the Fourier expansion of $f$
in $\CH_N$ with respect to the inner product $\la \cdot, \cdot \ra_N$, which
can be expressed as
\begin{equation} \label{interpolation}
 \CI_N f (x) = \sum_{k \in \Lambda_N} f(B^{-\tr}k)\Psi_{\Omega_B}^A(x-B^{-\tr}k),
 \qquad f \in C(\Omega_A),
\end{equation}
where
\begin{equation} \label{ell}
  \Psi_{\Omega_B}^A (x) = \frac{1}{|\det (N)|} \sum_{j \in \Lambda_{N}^{\dag}}
       \e^{2 \pi i  j^\tr A^{-1} x}.
\end{equation}

\begin{thm} \label{prop:interpolation}
Let $A, B$ and $N$ be as in Definition \ref{def:N}. Then
$\CI_N $
is the unique
interpolation operator
on $N$ in $\CH_N$; that is,
$$
    \CI_N f (B^{-\tr}j ) = f (B^{-\tr}j),\qquad  \forall j \in \Lambda_N.
$$
\end{thm}

In particular, $\# \Lambda_N = \# \Lambda_{N^\tr} = |\det(N)|$. The cubature
formula and the Lagrange interpolation are for functions that are periodic with
respect to the lattice $A \ZZ^d$, which are functions satisfying
$$
      f (x + A k) = f(x) \qquad \hbox{for all $k \in \ZZ^d$}.
$$
The function $x \mapsto \e^{2\pi i k^\tr A^{-1} x}$ is periodic with respect to
the lattice $A\ZZ^d$.

\subsection{Classical discrete Fourier analysis} We deduce the classical result
on the plane (cf. \cite{DM, Z}) from the general theory described above. As mentioned in the
introduction, we shall limit our consideration to cubature formulas. The result
hints at what is possible in the non-classical cases in the rest of the paper.

For $n \in \NN$, let $A = I$, the identity matrix, and $B = 2n I$. Then $N = B^\tr A = 2n I$
has all integer entries. Let $\Omega_A = [-\frac12, \frac12)^2$, which tiles $\RR^2$
with $\ZZ^2$ pointwisely and without overlapping. We shall write
$\Lambda_n$, $\Lambda_n^\dag$, $\CH_n$ in place of $\Lambda_N$,
$\Lambda_N^\dag$, $\CH_N$. Then
$$
\Lambda_n = \Lambda_n^\dag =\{k \in \ZZ^2: k \in [-n,n)^2 \}  \quad \hbox{and} \quad
\CH_n = \sspan\{\e^{2\pi i k \cdot x}: k\in \Lambda_n^\dag \}.
$$
It is clear that $\# \Lambda_n = (2n)^2$. The equation \eqref{c-d-inner} in this setting becomes
\begin{equation} \label{classical_ip}
\int_{[-\frac12,\frac12]^2} f(x) \overline{g(x)}dx = \frac{1}{4n^2}
     \sum_{k_1=-n}^{n-1}\sum_{k_2=-n}^{n-1}
   f(\tfrac{k_1}{2n}, \tfrac{k_2}{2n}) \overline{g(\tfrac{k_1}{2n}, \tfrac{k_2}{2n})}, \quad
    f, g \in \CH_n.
\end{equation}
To illustrate what can be done on cubature, we state the results in stages.

{\bf Stage 1.} It is easy to see that \eqref{classical_ip} yields a cubature formula
\begin{align} \label{cuba-SS1}
 \int_{[-\frac12,\frac12]^2} f(x) dx = \frac{1}{4 n^2} \sum_{k_1=-n}^{n-1}\sum_{k_2=-n}^{n-1}
     f(\tfrac{k_1}{2n}, \tfrac{k_2}{2n}), \qquad \forall f  \in \CH_{2n-1}^*,
\end{align}
where $\CH_n^*: = \sspan \{ \e^{2 \pi i k \cdot x}: k \in [-n,n]^2\cap \ZZ^2\}$.
The set of nodes of cubature \eqref{cuba-SS1} is not symmetric on $[-\frac12,\frac12]^2$,
since it has points on only part of the boundary of the square. We would like to
have cubature whose nodes is symmetric on the square.

{\bf Stage 2.} We construct cubature formulas for $\CH_{2n-1}^*$ that have symmetric
nodes on the square. Such a cubature is invariant under sign changes in both variables.
We can in fact obtain two such formulas from \eqref{cuba-SS1}. The first one is obtained
upon using the periodicity of the functions in the sums in the right hand side,
\begin{align} \label{cuba-SS2a}
  \int_{[-\frac12,\frac12]^2} f(x) dx = \frac{1}{4 n^2} \sum_{k_1=-n}^{n}\sum_{k_2=-n}^{n}
    c_{k,n} f(\tfrac{k_1}{2n}, \tfrac{k_2}{2n}), \qquad \forall f  \in \CH_{2n-1}^*,
\end{align}
where $c_{k,n} = 1$ if $k \in (-n,n)^2$, $c_{k,n} = 1/2$ if either $k_1 = \pm n$
or $k_2 = \pm n$ but not both, and $c_{k,n} = 1/4$ if $k = (\pm n, \pm n)$. The
second one is obtained by applying \eqref{cuba-SS1} to the function
$f(\cdot + \frac{1}{4n})$ and using the periodicity of $f$ in the integral,
\begin{align} \label{cuba-SS2b}
  \int_{[-\frac12,\frac12]^2} f(x) dx = \frac{1}{4 n^2} \sum_{k_1=-n}^{n-1}\sum_{k_2=-n}^{n-1}
      f(\tfrac{k_1+\frac12}{2n}, \tfrac{k_2+\frac12}{2n}), \qquad f  \in \CH_{2n-1}^*.
\end{align}

The fact that the set of nodes in either \eqref{cuba-SS2a} or \eqref{cuba-SS2b}
is invariant under the group $\ZZ_2^2$ (sign changes) allows us to derive cubature
formulas for product cosine and produce sine functions.
Let $\CTC_n : = \sspan \{\cos 2 \pi k_1 x_1 \cos 2\pi k_2 x_2:  0 \le k_1,k_2 \le n\}$
and $\CTS_n : = \sspan \{\sin 2 \pi k_1 x_1 \sin 2\pi k_2 x_2:  1 \le k_1,k_2 \le n\}$,
which consist of functions in $\CH_n^*$ that are invariant or anti-invariant under $\ZZ_2^2$,
respectively.

{\bf Stage 3.}  Restricting \eqref{cuba-SS2b} to $\CTC_{2n-1}$, we obtain a trigonometric
cubature,
\begin{equation}\label{cuba-SS3}
\int_{[0,\frac12]^2} f(x) dx = \frac{1}{ 4n^2} \sum_{k_1=0}^{n-1} \sum_{k_2=0}^{n-1}
    f(\tfrac{k_1+\frac12}{2n}, \tfrac{k_2+\frac12}{2n}), \qquad  \forall f  \in \CTC_{2n-1},
\end{equation}
whereas restricting \eqref{cuba-SS2a} to $\CTC_{2n-1}$ gives another trigonometric
cubature for $\CTC_{2n-1}$. Furthermore, restricting \eqref{cuba-SS2a} or
\eqref{cuba-SS2b} to $\CTS_{2n-1}$ leads to cubature for $\CTS_{2n-1}$.

The Chebyshev polynomials of the first and the second kind are defined, respectively,
by $T_n(\xi)=\cos 2\pi n \theta$ and $U_n(\xi)= \sin 2 \pi (n+1) \theta /\sin 2\pi \theta$, where
$\xi = \cos 2 \pi \theta$ with $0 \le \theta \le 1/2$. These are orthogonal polynomials with respect
to $w_0(\xi):= (1-\xi^2)^{-1/2}$ and $w_1(\xi):=(1-\xi^2)^{1/2}$, respectively, on $[-1,1]$.
Consequently, under the change of variables
\begin{equation}\label{SS:x-y}
       (x_1,x_2) \mapsto  (y_1, y_2) =  (\cos 2 \pi x_1,\cos 2 \pi x_2) \in [-1,1]^2,
\end{equation}
the space $\CTC_n$ is mapped into the product space $\Pi_n \times \Pi_n$, where
$\Pi_n$ denotes the space of algebraic polynomials of one variable, and $\CTS_n$
is mapped into $\Pi_{n-1} \times \Pi_{n-1}$.

{\bf Stage 4.} Under the map $x \mapsto y$ of \eqref{SS:x-y}, the cubature
\eqref{cuba-SS3} becomes
\begin{equation}\label{cuba-SS4}
     \frac{1}{\pi^2}\int_{[-1,1]^2} f(y)W_0(y) dy  =
    \frac{1}{ n^2} \sum_{k_1=0}^{n-1} \sum_{k_2=0}^{n-1}
        f(\cos \tfrac{\pi (2k_1+1)}{2 n}, \cos \tfrac{ \pi (2k_2+1)}{2n}),
\end{equation}
for $f  \in \Pi_{2n-1} \times \Pi_{2n-1}$, where $W_0(y) =w_0(y_1)w_0(y_2)$, which
is in fact the product Chebyshev-Gauss cubature of the first kind. Applying the same
procedure on the cubature \eqref{cuba-SS2a}, we obtain the product
Chebyshev-Gauss-Lobatto cubature. Furthermore, if we apply this procedure on
the cubature for $\CTS_{2n-1}$ that were mentioned in Stage 3, we obtain the
product cubature for the product Chebyshev weight $W_1(y) = w_1(y_1)w_1(y_2)$
of the second kind.


\section{Discrete Fourier analysis on planer domains}
\setcounter{equation}{0}

We now apply the general theory in Section 2.1 on the non-classical choices
of lattices. The guideline of our choices is the program that we outlined for the
classical case in subsection 2.2. We list the cases according to the shapes
of the fundamental domains of lattices. For example, the classical case in Section 2.2
is Square-Square. The main ones that we consider are the regular domains such
as square, rhombus, and regular hexagon, which are depicted below.
\begin{figure}[h]
\centering
\includegraphics[width=0.3\textwidth]{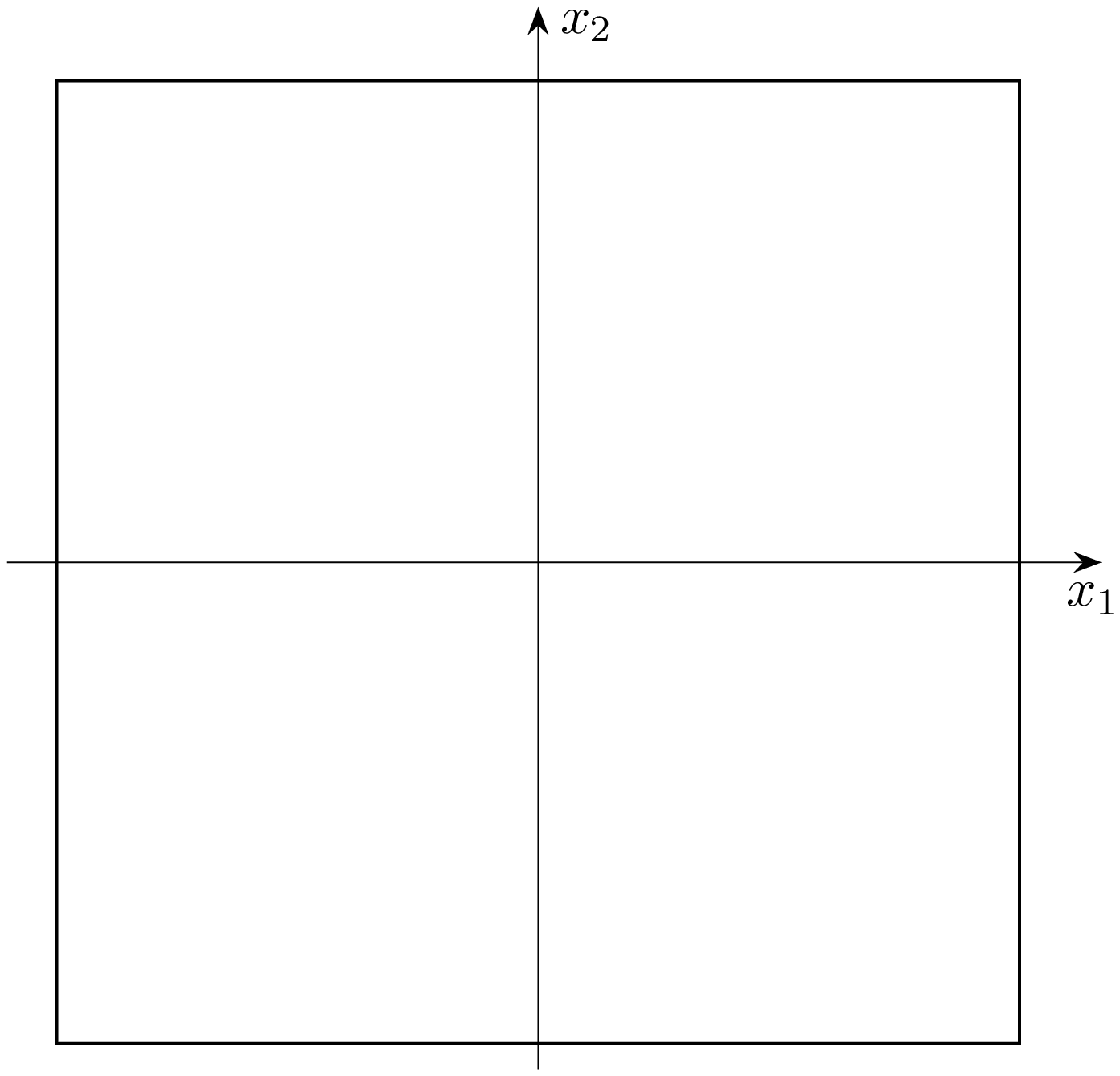}
\includegraphics[width=0.3\textwidth]{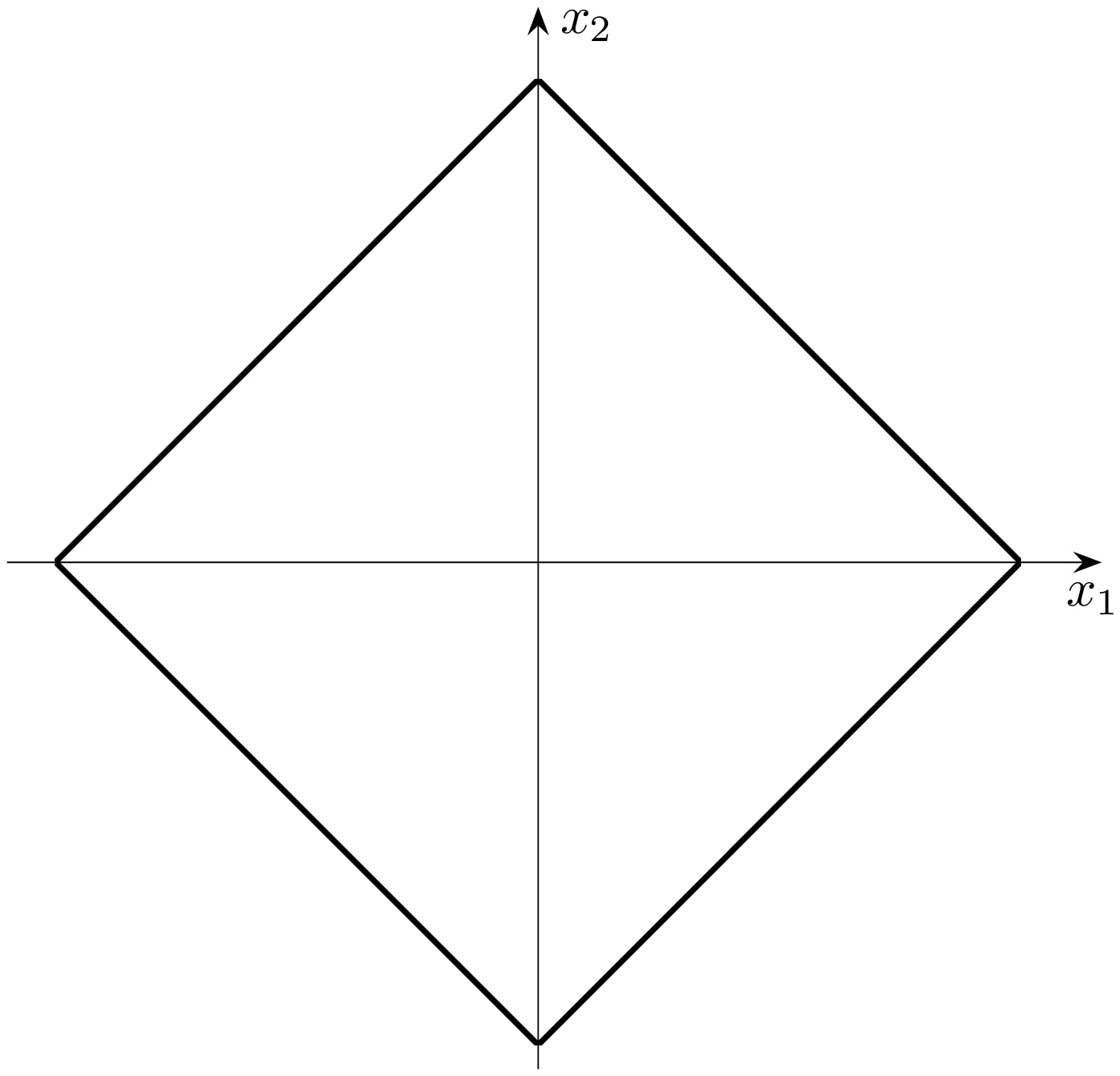}
\includegraphics[width=0.3\textwidth]{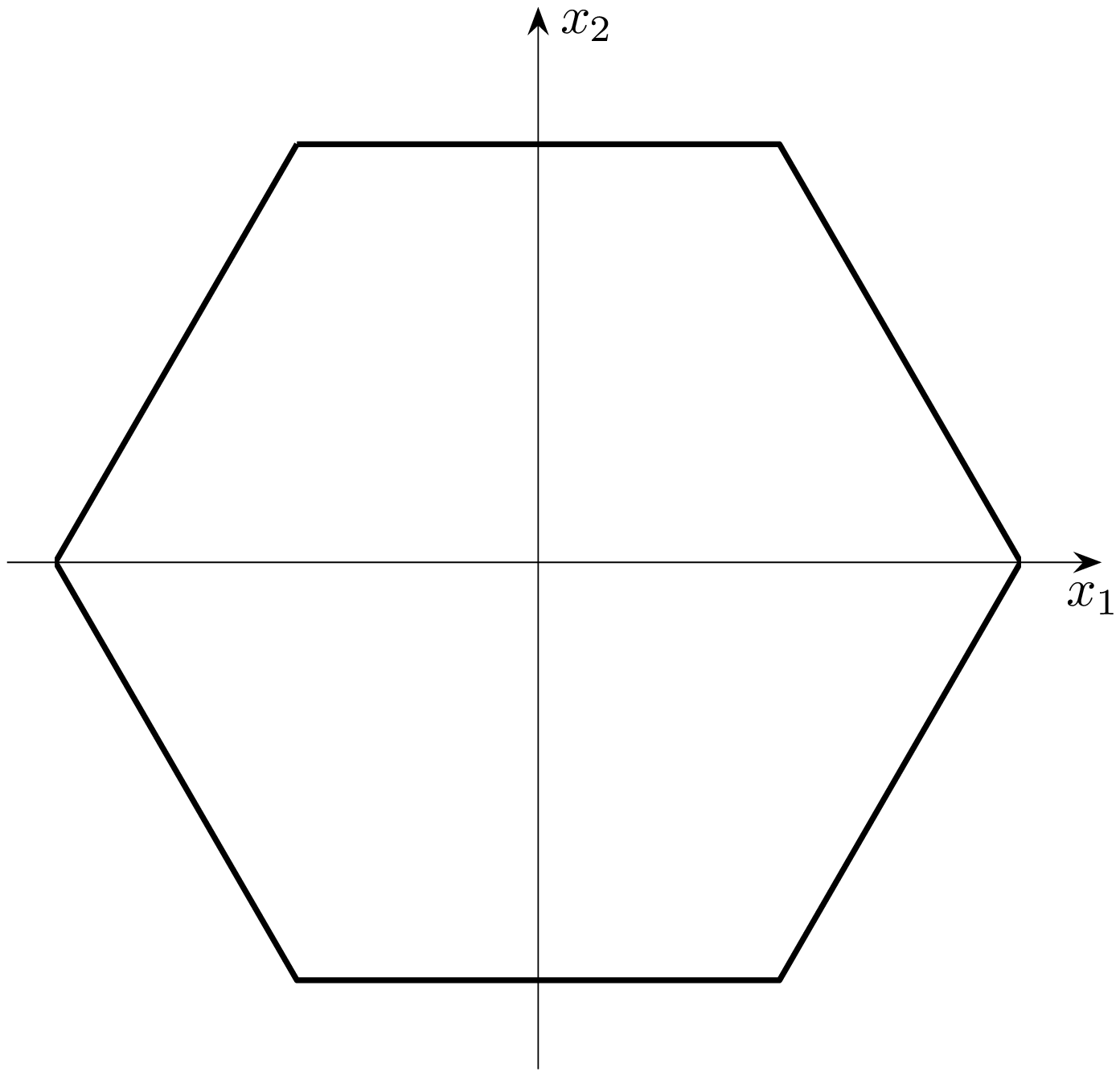}
\caption{Square, rhombus, regular hexagon}
\end{figure}
\subsection{Square-Rhombus} In this case we choose $A = I$ with $\Omega_A
= [ -\frac12, \frac 12)^2$ being the square and choose $B = n R$, where $R \ZZ^2$
has rhombic as its fundamental domain,
$$
 R: =   \left[ \begin{matrix} 1 & 1\\ -1 & 1 \end{matrix} \right] \quad
  \hbox{with}\quad  \Omega_B = \{x \in \RR^2 : -n \le x_1+x_2 < n, \,\, -n \le x_2 - x_1 < n\}.
$$
This case was studied in \cite{LSX2}. We shall be brief. Here $\Lambda_N =
\Lambda_N^\dag  = :\Lambda_n$, where
$$
 \Lambda_n  =  \{j \in \ZZ^2:  -n \le j_2 \pm j_1 <n\} \quad\hbox{and}\quad
    \Lambda_n^*: =  \{j \in \ZZ^2:  -n \le j_2 \pm j_1 \le n\}.
$$
The set $\Lambda_n$ is not symmetric on $[-n,n]^2$ but $\Lambda_n^*$ is.
The cardinality of $\Lambda_n$ is $|\Lambda_n| = 2 n^2$. We follow the program in
Section 2.2: In Stage 1 we deduce a cubature from Theorem \ref{thm:df}, which
has nodes indexed by $\Lambda_n$, then in Stage 2 we derive a cubature by
periodicity that has nodes indexed by $\Lambda_n^*$. By considering functions that
are even in both variables, we deduce in Stage 3 a cubature for trigonometric polynomials,
which we state as follows.  Changing variables from $j$ to $k = 2n B^{-\tr}j$, or
$k_1 =j_1+j_2$ and $k_2 = j_2 - j_1$,  it follows easily that $j \in \Lambda_n^*$ is
equivalent to
$$
   k  \in X_n^* =
   \left \{ 2 k: - \tfrac{n}{2} \le k_1, k_2 \le \tfrac{n}{2}\}   \cup  \{ 2 k+1:
      - \tfrac{n+1}{2}  \le k_1, k_2 \le \tfrac{n-1}{2}\right\}\!.
$$
Let $X_n^\circ$, $X_n^e$ and $X_n^v$ denote the set of points in $X_n^*$ that lie
in the interior, the edges excluding corners, and the corners of $[-n,n]^2$, respectively.

Throughout the rest of the paper, we will adopt the convention that $X^\circ$, $X^e$
and $X^v$ are subsets of $X$ defined as above, whenever the domain to which the
interior, edges and corners relate to is clear.

\begin{thm} \label{thm:cuba1}
For $n \ge 2$, the cubature formula
\begin{equation} \label{cuba1}
   \int_{[-\frac12,\frac12]^2} f(x) dx =  \frac{1}{2n^2} \sum_{k \in X_n^*} c_k^{(n)} f(\tfrac{k}{2n})
     \quad\hbox{with} \quad
            c_k^{(n)} = \begin{cases} 1, & k \in X_n^\circ \\ \frac12, & k \in X_n^e \\
                                             \frac14, & k \in X_n^v   \end{cases}
\end{equation}
 is exact for $f \in \CT_{2n-1}^*$, where  $\CT_m^* : = \mathrm{span}
      \left\{\e^{2\pi i\, k\cdot x} : \ k \in \Lambda_{m}^* \right\}.$
\end{thm}

The index set $\Lambda_n^*$ is most suitable for dealing with algebraic polynomials.
In fact, under the change of variables $x \mapsto y$ in \eqref{SS:x-y}, the space
$\CTC_m^*$ becomes the space $\Pi_m^2$ of algebraic polynomial of total degree
$m$ and the cubature \eqref{cuba1} becomes a cubature for the product
Chebyshev weight $W_0(y)$ that is exact for $\Pi_{2n-1}^2$. Let
\begin{equation*} 
  \Xi_n: = \{(2k_1,{2 k_2}): 0 \le k_1,k_2 \le \tfrac{n}{2}\}
       \cup  \{({2k_1+1},{2k_2+1}): 0 \le k_1,k_2 \le \tfrac{n-1}{2}\}.
\end{equation*}
Then, in Stage 4, \eqref{cuba1} becomes the following:

\begin{thm}
Let $z_{k} = (\cos \frac{k_1 \pi}{n},\cos \frac{k_2 \pi}{n})$. Then the cubature below is exact for $f \in \Pi_{2n-1}^2$,
\begin{equation} \label{cubaT}
 \frac{1}{\pi^2} \int_{[-1,1]^2} f(y) W_0(y) dy
        = \frac{1}{2 n^2}\sum_{k \in \Xi_n} \lambda_k^{(n)} f(z_k),
        \quad
  \lambda_k^{(n)} : = \begin{cases} 4, & k \in \Xi_n^\circ, \\
     2, & k \in \Xi_n^e, \\
     1, & k \in \Xi_n^v. \end{cases}
\end{equation}
\end{thm}

The cardinality of $\Xi_n$ is $|\Xi_n| = \frac{n(n+1)}{2} + \lfloor \frac{n}{2} \rfloor +1$,
which is just one more than the theoretic lower bound for all such cubature (\cite{DX, M}).
The formula \eqref{cubaT} first appeared in \cite{X94}, where it is constructed by
considering the common zeros of orthogonal polynomials of two variables; see also
\cite{MP}. We can also derive similarly cubature for the product Chebyshev weight
$W_1(y)$ of the second kind.

The Lagrange interpolation polynomials based on the points in $\Xi_n$ were
defined and studied in \cite{X96}. The result there has also been recovered in \cite{LSX2},
with \eqref{interpolation} as the starting point, by utilizing the discrete Fourier analysis.

\subsection{Rhombic-Square} In this case we choose $A = R$ with fundamental
domain $\Omega_R = \{x \in \RR^2: -1 \le x_2 \pm x_1 < 1\}$, the rhombic, and
$B = n I$. Again write $\Lambda_n$ ... in place of $\Lambda_N$ ....
It is then easy to verify that $\Lambda_n = \Lambda_n^\dag$ with
$$
  \Lambda_n = \{j\in \ZZ^2: - n \le k_2 \pm k_1 < n\} \quad \hbox{and}\quad
  \Lambda_n^* : = \{j\in \ZZ^2: - n \le k_2 \pm k_1 \le n\}.
$$
Furthermore, the space of exponential functions $\CH_n$ is given by
$$
  \CH_n = \{\e^{\pi i ( (k_1+k_2)x_1+(k_2-k_1)x_2)}: k \in \Lambda_n\}
$$
and $\CH_n^*$ is likewise defined in terms of $\Lambda_n^*$. Changing variables
shows that
$$
 \CH_n^* :  = \{\e^{\pi i (j_1 x_1+ j_2 x_2)}:  -n \le j_1, j_2 \le n, \,\, j_1 \equiv j_2 \pmod 2\}.
$$
Following the program in Section 2.1, it is easy to see that the cubature in Stage 2 that
has symmetric set of nodes, indexed by $\Lambda_n^*$, takes the form
\begin{equation}\label{cuba-RS}
  \frac{1}{2} \int_{\Omega_R} f(x) dx = \frac{1}{2n^2} \sum_{k \in \Lambda_n^*} c_{k,n}
      f( \tfrac{k}{n}), \quad f \in \CH_{2n-1}^*, \quad c_{k,n} =
        \begin{cases} 1, & k \in \Lambda_n^\circ, \\  \frac12, & k \in \Lambda_n^e,
                   \\ \frac14, & k \in \Lambda_n^v. \end{cases}
\end{equation}
The subspace of functions in $\CH_n^*$ that are even in both variables becomes
\begin{equation} \label{eq:Tn}
\CT_n :=\sspan \{\cos \pi j_1x_1\cos\pi j_2x_2: 0 \le j_1,j_2 \le n, \,\, j_1 \equiv j_2 \pmod 2\}.
\end{equation}
For functions in $\CT_{2n-1}$, we only need to consider the triangle $T_R :=
\{x: 0 \le x_1,x_2, \,\ x_1+x_2 \le 1\}$. Thus, in Stage 3,  cubature \eqref{cuba-RS}
becomes
\begin{equation}\label{cuba-RS2}
 2  \int_{T_R} f(x) dx = \frac{1}{2n^2} \sum_{k \in \Xi_n} \lambda_{k,n}
      f( \tfrac{k}{n}), \qquad f \in \CT_{2n-1},
\end{equation}
where $T_R$ is the triangular domain
$\Xi_n =\{(k_1,k_2): 0 \le k_1, k_2,  k_1 + k_2 \le n\}$, and
$\lambda_{k,n} = 4$ if $k \in \Xi_n^\circ$, $\lambda_{k,n} = 2$
if $k \in \Xi_n^e$, $\lambda_{(0,0),n} =1$, and $\lambda_{(n,0),n} = \lambda_{(0,n),n} =1/2$.

Under the mapping $x \mapsto y=(\cos\pi x_1, \cos \pi x_2)$, 
the boundary $x_1+ x_2 =1$
of the triangle $T_R$ is mapped onto $y_1 + y_2 =0$, so that $T_R$ is mapped
onto the triangle $T_S = \{(y_,y_2): -1 \le y_1,y_2 \le 1, y_1 + y_2 \ge 0\}$, which is
half of the square $[-1,1]^2$. The cubature \eqref{cuba-RS2} in Stage 4 becomes
a cubature with respect to the product Chebyshev weight $W_0$ over $T_S$ that
is exact for the subspace of polynomials $\Pi_{2n-1}^* = \{T_{k_1}(x_1)T_{k_2}(x_2):
0 \le k_1,k_2 \le 2 n-1, k_1 \equiv k_2\pmod 2 \}$, the image of $\CT_{2n-1}$
under the same mapping. Since $\Pi_n^*$ does not contain polynomials of
total degree, we shall not write this cubature explicitly out. It is easy to see,
however, that this cubature is in fact half of the product Chebyshev-Gaussian-Lobatto
cubature, in the sense that its domain is half and it is exact for half of the polynomials
of the latter cubature.

\subsection{Rhombic-Rhombic} Here we choose $A = R$ and $B = n R^{-\tr}
= \frac{n}{2} A$, so that $N = B^{\tr} A = nI$ have integer entries. Then
$\Omega_A = \Omega_R$ as in the previous case. Again denote $\Lambda_N$,
... by $\Lambda_n$, ... . It is easy to see that $\Lambda_n = \Lambda_n^\dag$ with
$$
\Lambda_n   = \{j \in \ZZ^2: - \tfrac{n}{2} \le - j_1, j_2 < \tfrac{n}{2}\} \quad\hbox{and} \quad
\Lambda_n^* := \{j \in \ZZ^2: - \tfrac{n}{2} \le  j_1, j_2 \le \tfrac{n}{2}\}.
$$
Moreover, the space of exponential functions $\CH_n$ is given by, as in Section 3.2,
$$
  \CH_n = \{\e^{ \pi i ( (k_1+k_2)x_1+(k_2-k_1)x_2)}: k \in \Lambda_n\}
$$
and $\CH_n^*$ is likewise defined with $\Lambda_n$ replaced by $\Lambda_n^*$.
In this case, the cubature derived from Theorem \ref{thm:df}, in Stage 1,
takes the form
\begin{equation}\label{cuba-RR}
  \frac{1}{2} \int_{\Omega_R} f(x) dx = \frac{1}{n^2} \sum_{k \in \Lambda_n}
      f(\tfrac{k_1+k_2}{n}, \tfrac{k_2-k_1}{n}), \quad \forall f \in \CH_{2n-1}^*.
\end{equation}
The set of nodes of this cubature is on $\Omega_R$, and it contains no points on
the boundary of $\Omega_R$ when $n$ is an odd integer, whereas it contains
points on half of the boundary of $\Omega_R$ when $n$ is an even integer.
In the latter case, we can again derive a cubature, exact for $\CH_{2n-1}^*$,
that has notes indexed by $\Lambda_n^*$ as in Stage 2. Let us consider,
however, only the case of $n$ being an odd integer below. As can be seen upon
changing variables $j_1 = k_1+k_2$ and $j_2 = k_2-k_1$, the subspace of
functions in $\CH_n^*$ that are even in both variables is exactly $\CT_n$ in
\eqref{eq:Tn}. Thus, just like in the case of Rhombic-Square, restricting
\eqref{cuba-RR} to functions in $\CH_{2n-1}^*$ that are even
in both variables, we deduce a cubature of Stage 3 on the triangle $T_R$,
\begin{equation}\label{cuba-RR2}
 2  \int_{T_R} f(x) dx = \frac{1}{n^2} \sum_{k \in \Xi_n} \lambda_{k,n}
      f( \tfrac{k_1+k_2}{n}, \tfrac{k_2-k_1}{n}), \qquad \forall f \in \CT_{2n-1},
\end{equation}
where $\Xi_n = \{(k_1,k_2): 0 \le k_1,k_2, k_1+k_2 \le \frac{n}{2} \}$;
$\lambda_{k,n} = 4$ if $k \in \Xi_n^\circ$,
$\lambda_{k,n} = 2$ if   either $k_1 =0$ or $k_2=0$ or $k_1 + k_2 = \frac{n}{2}$ but not both (i.e.,$k \in \Xi_n^e$),
and $\lambda_{(0,0),n} =1$, $\lambda_{(0,n),n} =1$ .

Finally, under the mapping $x \mapsto y = (\cos \pi x_1, \cos \pi x_2)$, the cubature
\eqref{cuba-RR2} becomes a cubature with respect to the product Chebyshev
weight $W_0$ over the triangle domain $T_S$ for the polynomial subspace
$\Pi_{2n-1}^*$ defined in the previous subsection. This cubature is exactly
half of the algebraic cubature  in the Square-Rhombic case.

\subsection{Hexagon-Hexagon}
In this case we choose $A = H$ and $B  =  \frac{n}{2} H$, where
$$
   H =  \begin{pmatrix} \sqrt{3} & 0\\ -1 & 2\end{pmatrix} \quad \hbox{with} \quad
   \Omega_H=\left\{x\in \RR^2:\  -1\leq x_2, \tfrac{\sqrt{3} x_1}{2} \pm
\tfrac{x_2}{2}  < 1 \right\}.
$$
This case was studied in \cite{LSX1}. We shall be brief, but recall necessary definitions
that are needed in the following subsection. As shown in \cite{LSX1,Sun}, it is more
convenient to use homogeneous coordinates  $(t_1,t_2,t_3)$ defined by
\begin{align}\label{coordinates}
\begin{pmatrix}t_1 \\ t_2 \\ t_3 \end{pmatrix} =
 \begin{pmatrix}  \frac{\sqrt{3}}{2} & -\frac12 \\
                          0 &  1\\
                  -\frac{\sqrt{3}}{2} & -\frac12
 \end{pmatrix}  \begin{pmatrix}x_1 \\ x_2 \end{pmatrix} := E x,
\end{align}
which satisfies $t_1 + t_2 +t_3 =0$. We adopt the convention of using bold letters,
such as $\tb$ to denote points in homogeneous coordinates. We define by
$$
    \RR_H^3 : = \{\tb = (t_1,t_2,t_3)\in \RR^3: t_1+t_2 +t_3 =0\} \quad \hbox{and} \quad
     \HH := \ZZ^3 \cap \RR^3_H
$$
the spaces of points and integers in homogeneous coordinates, respectively. In such
coordinates, the hexagon $\Omega_H$ becomes
\begin{align*}
\Omega=\left\{\tb \in \RR_H^3:\  -1\le  t_1,t_2,-t_3<1 \right\},
\end{align*}
which is the intersection of the plane $t_1+t_2+t_3=0$ with the cube $[-1,1]^3$.
The index sets $\Lambda_n$  and $\Lambda_n^\dag$ satisfy $\Lambda_n =
\Lambda_n^\dag = \HH_n$, where
$$
  \HH_n :=  \{\jb \in \HH: -n \le j_1,j_2, - j_3 < n \} \quad \hbox{and}\quad
   \HH_n^* :=  \{\jb \in \HH^*: -n \le j_1,j_2, - j_3 \le n \}.
$$
Furthermore, since, for $k = (k_1,k_2)$, $k^{\tr} H^{-1}x = \frac13 \kb \cdot \tb$ with
$\kb = (k_1,k_2,-k_1 - k_2)^\tr \in \HH$, the exponential functions and the space
$\CH_N$ become
$$
 \phi_\kb(\tb): = \e^{\frac{2 \pi i}{3}\kb^\tr \tb} \quad
       \hbox{and}\quad   \CH_n: =\{ \phi_{\kb}: \kb \in \HH_n\}.
$$
In homogeneous coordinates, $x \equiv y \pmod H$ becomes
$\tb \equiv \sb \pmod 3$, which is defined by $t_1-s_1 \equiv t_2-s_2 \equiv t_3-s_3
\pmod 3$, so that $f$ periodic in $H$, i.e. $f(x + H ) = f(x)$, becomes
$f(\tb) = f(\tb + \jb)$ whenever $\jb \equiv 0 \pmod 3$.

In this case, the cubature derived from Theorem \ref{thm:df} in Stage 1 has nodes
over $\{\frac{\jb}{n}: \kb \in \HH_n\}$, from which we derive another cubature, the set of
nodes of which is symmetric and indexed by $\HH_n^*$, as in Stage 2:

\begin{thm} \label{thm:HH1}
The following cubature is exact for $f\in \CH_{2n-1}^*$,
\begin{equation} \label{cuba-HH}
\frac{1}{|\Omega|} \int_{\Omega} f(\tb)d\tb  = \frac{1}{3n^2}
   \sum_{\jb\in \HH_n^*} c_\jb^{(n)} f(\tfrac{\jb}{n}), \quad
     c_{\jb}^{(n)} = \begin{cases} 1, & \jb \in \HH_n^\circ,\\
       \frac{1}{2},  & \jb \in \HH_n^e,\\ \frac{1}{3},    & \jb \in \HH_n^v,  \end{cases}
 \end{equation}
\end{thm}

The group of isometries of the hexagon lattice is generated by the reflections in the
edges of the equilateral triangles inside the regular hexagon, which is the reflection
group $\A_2$. By considering the invariant and anti-invariant functions under $\CA_2$
in the space $\CH_n$, we end up with functions that are analogues of cosine and sine
functions on an equilateral triangle, and the cubature \eqref{cuba-HH} becomes a
cubature on the triangle for such functions. To be more precise, we choose the triangle as
\begin{align} \label{Delta}
   \Delta := & \{(t_1,t_2,t_3) : t_1 + t_2 + t_3 =0,   0 \le t_1,  t_2, -t_3 \le 1\}.
\end{align}
The region $\Delta$ and its relative position in the hexagon are depicted
in Figure \ref{DeltaF}, where the points are labeled in homogeneous coordinates.
\begin{figure}[h]
\centering
\begin{minipage}{0.4\textwidth}\centering \includegraphics[width=1\textwidth]{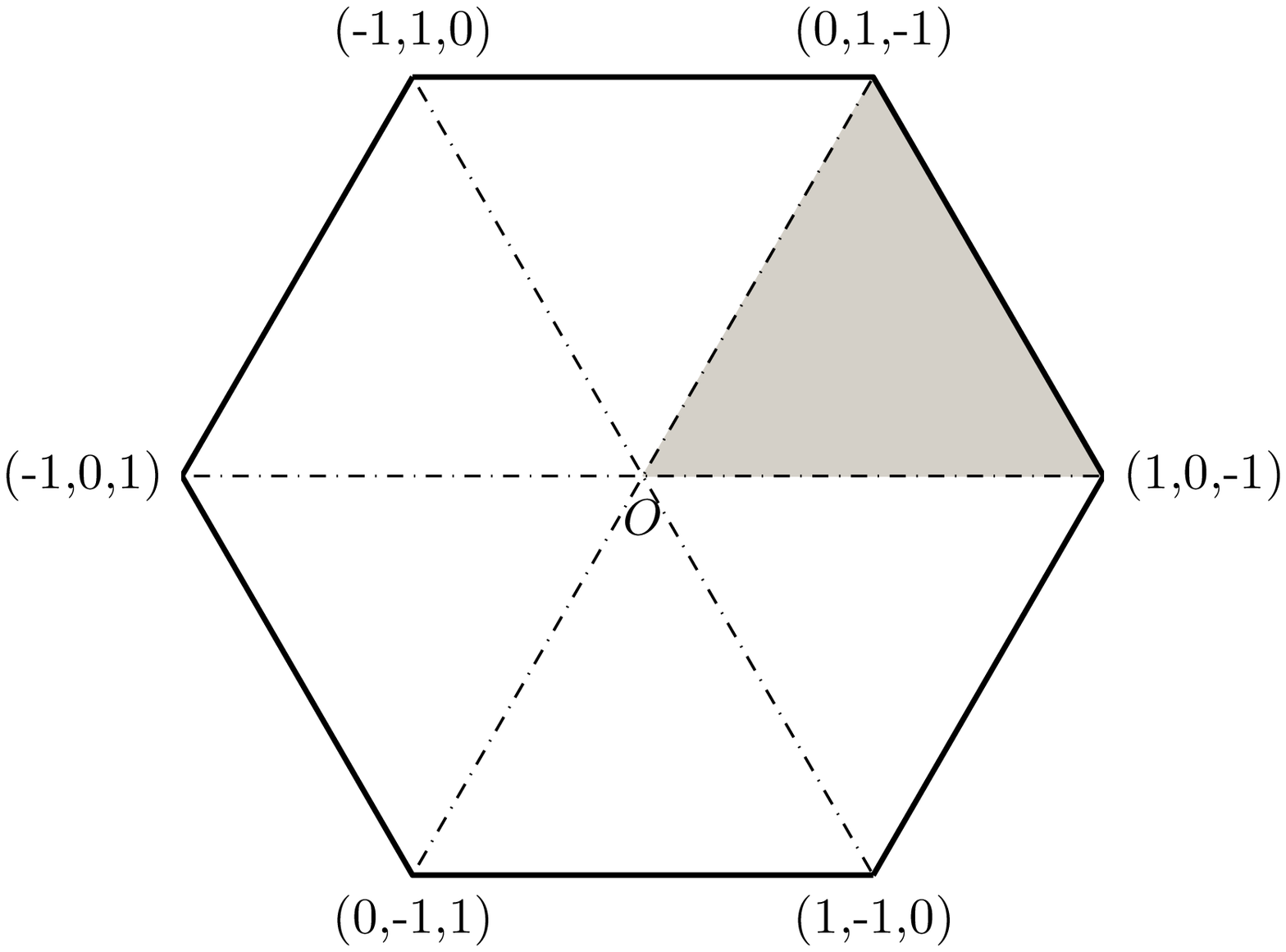}
\end{minipage}
\begin{minipage}{0.33\textwidth}\centering \includegraphics[width=1\textwidth]{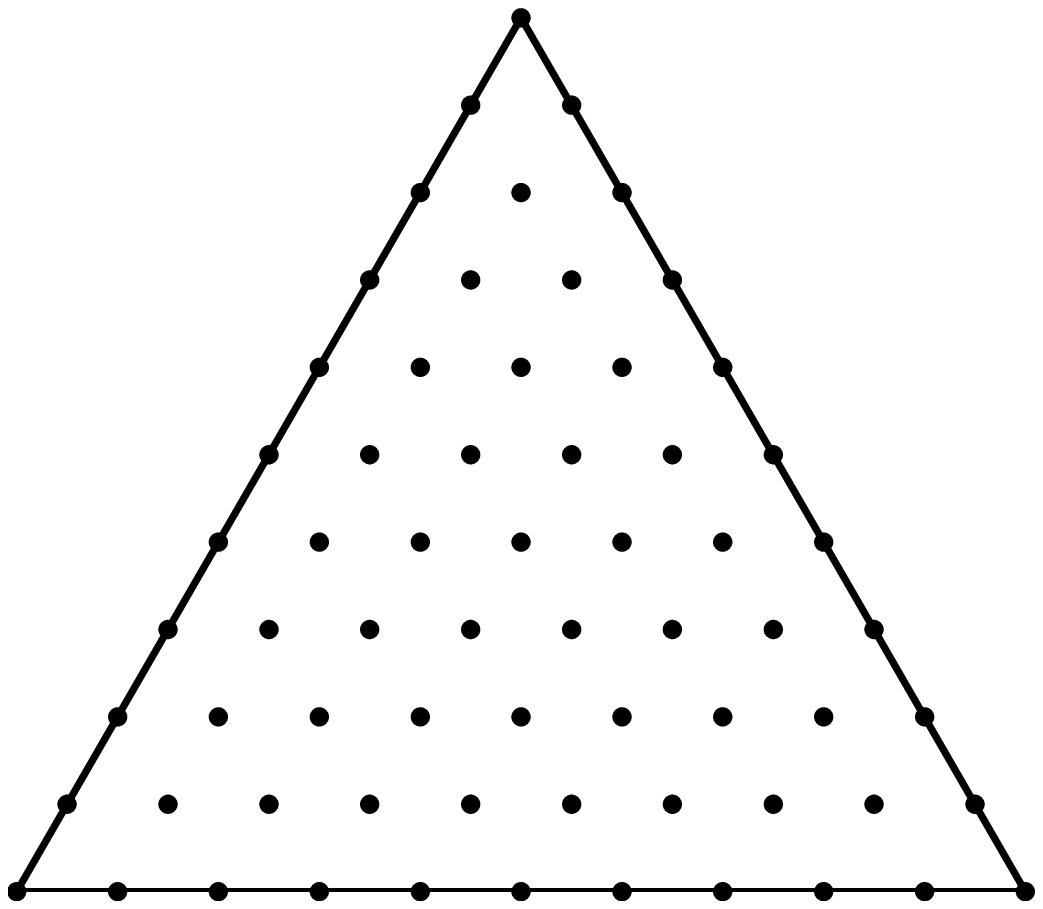}
\end{minipage}
\caption{The fundamental triangle of the regular hexagon.}
\label{DeltaF}
\end{figure}
The generalized cosine, $\TC_\kb$, and the generalized sine, $\TS_\kb$, are defined
in terms of
\begin{equation}\label{CP-pm}
  \CP^+ f(\tb) := \sum_{\sigma \in \CA_2} f(\tb \sigma) \quad \hbox{and}\quad
   \CP^- f(\tb) := \sum_{\sigma \in \CA_2} \sign(\sigma) f(\tb \sigma)
\end{equation}
as $\TC_\kb(\tb) := \CP^+ \phi_\kb(\tb)$ and $\TS_\kb(\tb) := \CP^- \phi_\kb(\tb)$,
respectively; more explicitly,
\begin{align}
 \TC_\kb(\tb) =&
 \frac{1}{3} \left[ \e^{\frac{i\pi}{3}(k_2-k_3)(t_2-t_3)}\cos k_1\pi t_1  \label{TC_cos}
  +    \e^{\frac{i\pi}{3}(k_2-k_3)(t_3-t_1)}\cos k_1\pi t_2 \right.\\
      & \left. +\e^{\frac{i\pi}{3}(k_2-k_3)(t_1-t_2)}\cos k_1\pi t_3\right],   \qquad \kb \in \Lambda,
         \notag\\
 \TS_\kb(\tb) =&
 \frac{1}{3} \left[
   \e^{\frac{i\pi}{3}(k_2-k_3)(t_2-t_3)}\sin k_1\pi t_1\label{TS_sin}
      +    \e^{\frac{i\pi}{3}(k_2-k_3)(t_3-t_1)}\sin k_1\pi t_2 \right.\\
       & \left.+\e^{\frac{i\pi}{3}(k_2-k_3)(t_1-t_2)}\sin k_1\pi t_3\right],
            \qquad \kb \in \Lambda^\circ, \notag
\end{align}
where  $\Lambda: = \{\kb \in \HH: k_1 \ge 0, k_2 \ge 0, k_3 \le 0\}$ and $\Lambda^\circ$ is the
interior of $\Lambda$. These functions are orthogonal with respect to the integral over
$\Delta$, and they are elements of $\CH_n$ that invariant and anti-invariant under $\CA_2$,
respectively. The cubature \eqref{cuba-HH} when restrict to invariant functions
becomes, as in Stage 3, the following:

\begin{thm} \label{thm:HH2}
Let $\CTC_n := \sspan\{ \TC_\kb: \kb \in \Lambda_n\}$. The cubature below is exact
for all $f \in \CTC_{2n-1}$,
\begin{equation}\label{cubaHH2}
  2  \int_{\Delta} f(t_1,t_2) dt_1 dt_2 = \frac{1}{3n^2}
      \sum_{j_1=0}^n\sum_{j_2=0}^{j_1}
          \lambda_\jb^{(n)} f(\tfrac{j_1}{n},\tfrac{j_2}{n}), \quad
          \lambda_\jb^{(n)} : =  \begin{cases}
       6, & \jb \in \Lambda_n^\circ,\\
       3 , & \jb \in \Lambda_n^\e,\\
       1, & \jb \in \Lambda_n^\ve.
       \end{cases}
\end{equation}
\end{thm}

The nodes of the cubature \eqref{cubaHH2} are equally spaced points in $\Delta$ (Figure 2).

The generalized cosine and sine functions can be mapped into algebraic polynomials of two variables under the following mapping,
\begin{align} \label{x-y}
\begin{split}
x & =  \tfrac{4}{3} \cos\tfrac{\pi}{3}(t_2-t_1) \cos\tfrac{\pi}{3}(t_3-t_2)
  \cos\tfrac{\pi}{3}(t_1-t_3)-\tfrac{1}{3}, \\
y & = \tfrac{4}{3} \sin\tfrac{\pi}{3}(t_2-t_1) \sin\tfrac{\pi}{3}(t_3-t_2)
     \sin\tfrac{\pi}{3}(t_1-t_3),
\end{split}
\end{align}
which are the real and imaginary part of $\TC_{0,1,-1}(\tb)$, the first non trivial generalized
cosine function. Under this mapping, we call the polynomials
\begin{align*}
T_k^m(x,y) :   = \TC_{k, m-k, - m}(\tb) \quad \hbox{and}\quad
 U_k^m(x,y) :   = \frac{\TS_{k+1,m-k+1,-m-2}(\tb)}{\TS_{1,1,-2}(\tb)},
\end{align*}
where $0 \le k \le m$, generalized Chebyshev polynomials of the first and the second
kind, respectively. They are algebraic polynomials of total degree $n$ and are
orthogonal polynomials with respect to the weight function $w_{-\frac12}(x,y)$ and
$w_{\frac12}(x,y)$, respectively, where $w_\a(x,y)$ is defined by
$$
  w_\alpha(x,y) = \frac{4^\alpha}{27^\alpha} \pi^{4\alpha}
      \left[-3(x^2+y^2+1)^2 + 8 (x^3-3x y^2)+4 \right]^\alpha,
$$
and the integral domain is the region $\Delta^*$ bounded by the Steiner's hypocycloid,
depicted in Figure 3, which is the region on which $w_\a(x,y)$ is positive.
\begin{figure}[h]
 \includegraphics[width=0.3\textwidth,height=0.3\textwidth]{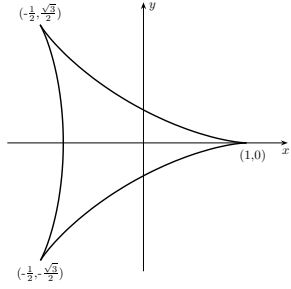}
 \caption{The region $\Delta^*$ bounded by Steiner's hypocycloid.}
\end{figure}
These polynomials were first studied in \cite{K}. As in Stage 4, the cubature
\eqref{cubaHH2} under the change of variable \eqref{x-y} becomes a cubature for
$w_{-\frac12}(x,y)dxdy$ on $\Delta^*$ that has $\dim \Pi_{n}^2$ nodes and is exact
for algebraic polynomials of degree $2n-1$. Furthermore, we can derive a cubature
from \eqref{cuba-HH} for anti-invariant functions in Stage 3, which becomes under
\eqref{x-y} a cubature for $w_\frac12(x,y)dxdy$ that has $\dim \Pi_{n-1}^2$ nodes
and is exact for algebraic polynomials of degree $2n-1$. The latter one provides an
example of a family of Gaussian cubature formulas, a rarity of only the second
example known so far (the first one appeared in \cite{SX}); see \cite{LSX1} for details.
We refer to \cite{DX, My, St} for the topic of Gaussian cubature.

We now address one question that was not addressed in \cite{LSX1}. Taking the
cue form the cubature \ref{cuba-SS2b} in the Square-Square case, we can apply the
cubature derived in Stage 1 on the functions $f(\tb+ \ab)$ and then use the hexagonal
periodicity of the integral to derive the following cubature in Stage 2,
\begin{equation} \label{cubaHH3}
  \frac{1}{|\Omega|} \int_\Omega f(\tb) d\tb =  \frac{1}{3n^2} \sum_{\jb \in \HH_n}
      f(\tfrac{\jb}{n} +\ab), \qquad f  \in \CH_{2n-1}^*,
\end{equation}
and hope to choose $\ab$ so that the set of nodes in \eqref{cubaHH3} is symmetric.
The question is if it is possible to find a $\ab$ so that the set of nodes
has full symmetry of $\CA_2$.

It is easy to see that if $\ab = (a_1,a_2,-a_1-a_2) \in \RR_H^3$ satisfies $|a_1|, |a_2| \le 1/n$,
then the set of nodes of \eqref{cubaHH3} will be inside the hexagon $\Omega$, although
not symmetric for most of the choices. The two cases that offer the most symmetry are
$$
  \ab_1: = (\tfrac1{3n}, \tfrac1{3n}, - \tfrac2{3n}) \quad \hbox{and} \quad
  \ab_2 := (-\tfrac1{3n}, -\tfrac1{3n},  \tfrac2{3n}),
$$
where, when $\ab_2$ is used, we need to use the periodicity of $f$ (or congruent relation 
with respect to $H$) to make sure that all points in \eqref{cubaHH3} are in $\Omega$. 
Neither of these two choices, however, offer complete symmetry under the group $\CA_2$.
In Figure 4, we depict the set of points resulted from these two choices.
\begin{figure}[h]
\hfill\includegraphics[width=0.35\textwidth]{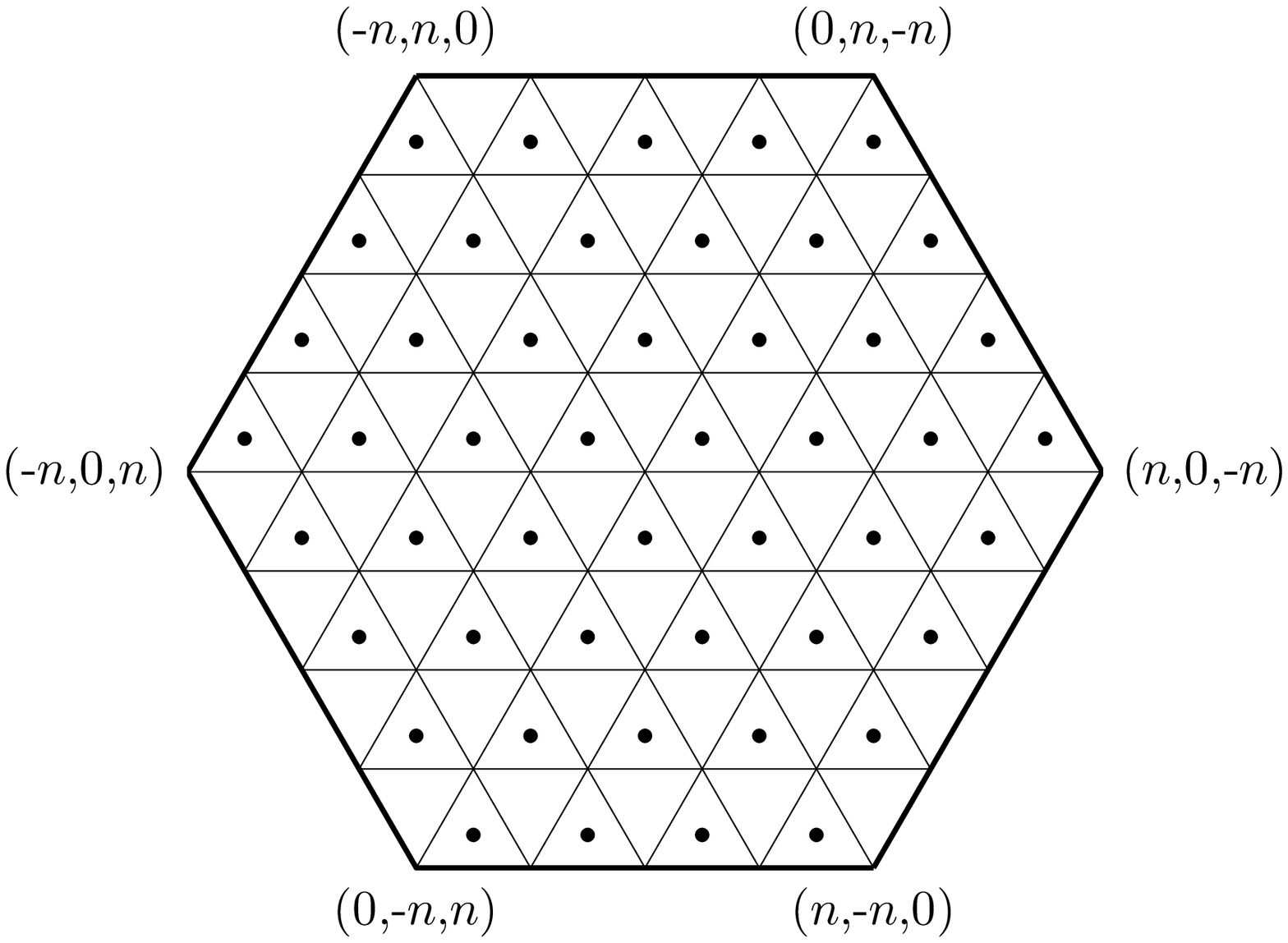}
\hfill\includegraphics[width=0.35\textwidth]{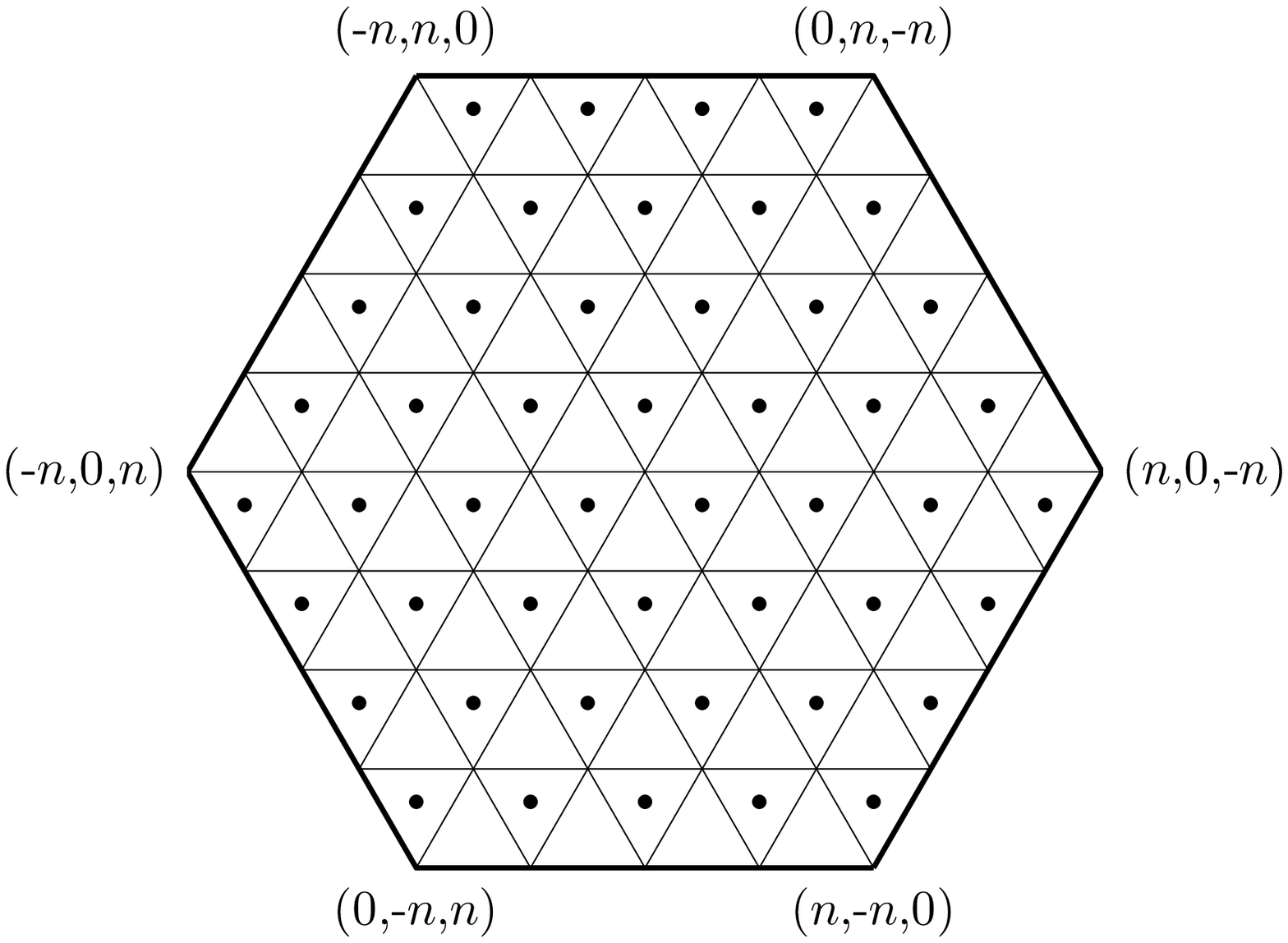}
\hspace*{\fill}
\caption{Nodes of \eqref{cubaHH3} with $\ab_1$ (left) and $\ab_2$ (right)}
\end{figure}
Each set of the points is invariant under a subgroup of $\CA_2$ of three rotations, but
neither is invariant under the group $\CA_2$. As a result, we cannot restrict the cubature
\eqref{cubaHH3} with either $\ab_1$ or $\ab_2$ to the generalized cosine or sine functions
in hopes of obtaining new cubature on the triangle in Stage 3, in contrast to
Square-Square case.

The interpolation on the hexagon and on the triangle were studied in \cite{LSX1}. In
particular, we have Lagrange interpolation based on equally space points on the triangle
$\Delta$, which enjoys a closed formula in trigonometric functions and has Lebesgue
constant in the order of $(\log n)^2$. One can also consider approximation on the hexagon
and the triangle (\cite{X09}) for functions that are periodic in $H$.

\subsection{Hexagon-Hexagon Transpose}
Here we choose $A = H$, the matrix for the hexagon lattice, and choose $B = n H^{-\tr}$
with $n \in \ZZ$,  so that $N = B^\tr A = n I$ has all integer entries. The fundamental
domain of the lattice $L_B$ is given by
$$
  \Omega_B = \left\{ x\in \RR^2: -\tfrac{n}{2\sqrt{3}} \le x_1,
        \tfrac{\sqrt{3}}{2} x_2 \pm \tfrac12 x_1 < \tfrac{n}{2\sqrt{3}}  \right\}.
$$
\subsubsection{Cubature}
It is again convenient to use homogeneous coordinates as defined in the previous
subsection. The $\Omega_B$ is the regular hexagon in Fig. 1 rotated by $90^\circ$,
as depicted in Figure 2, in which the right hand figure is labeled in homogeneous
coordinates.
\begin{figure}[h]
\hfill\includegraphics[width=0.3\textwidth]{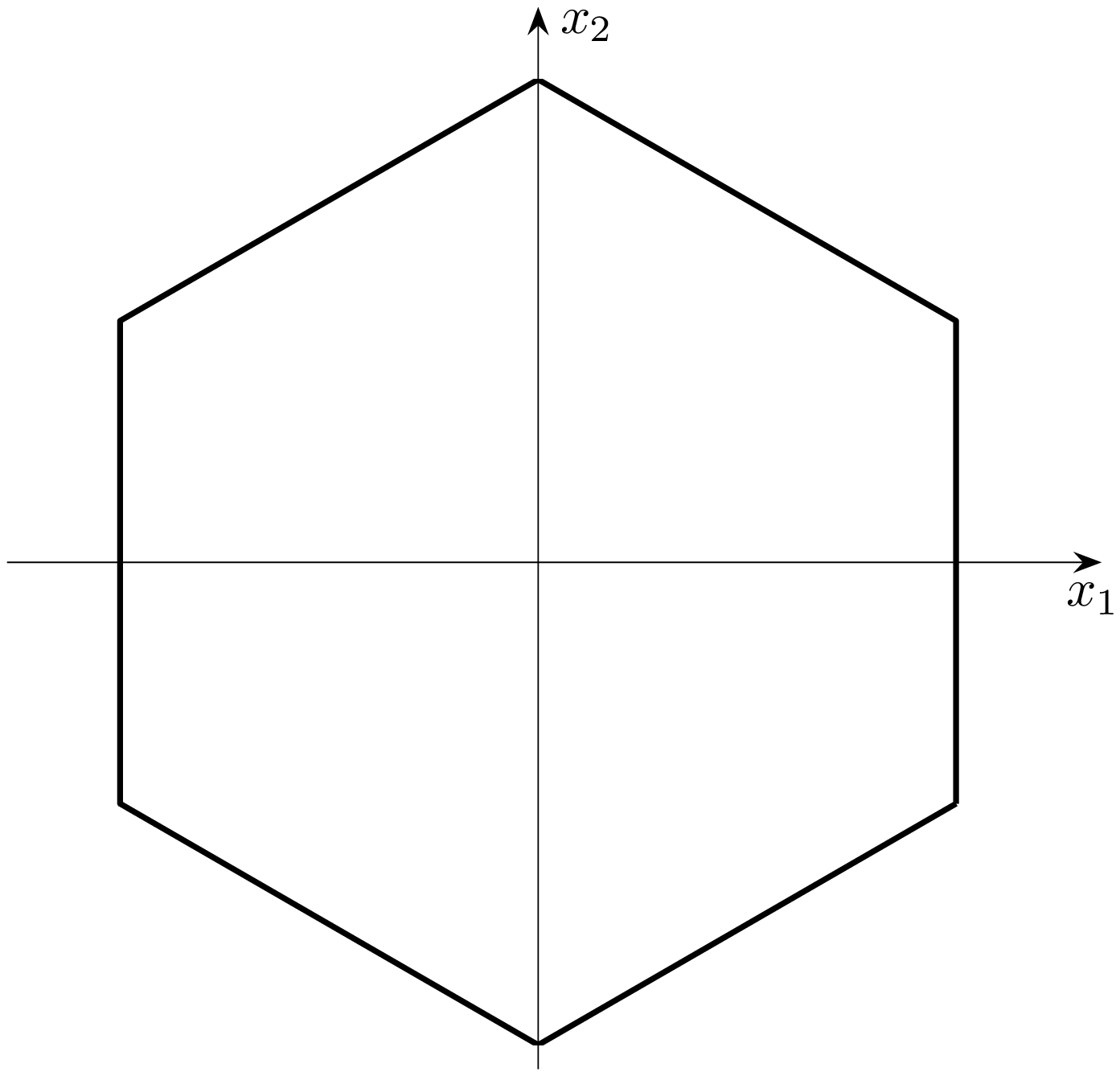}
\hfill\includegraphics[width=0.35\textwidth]{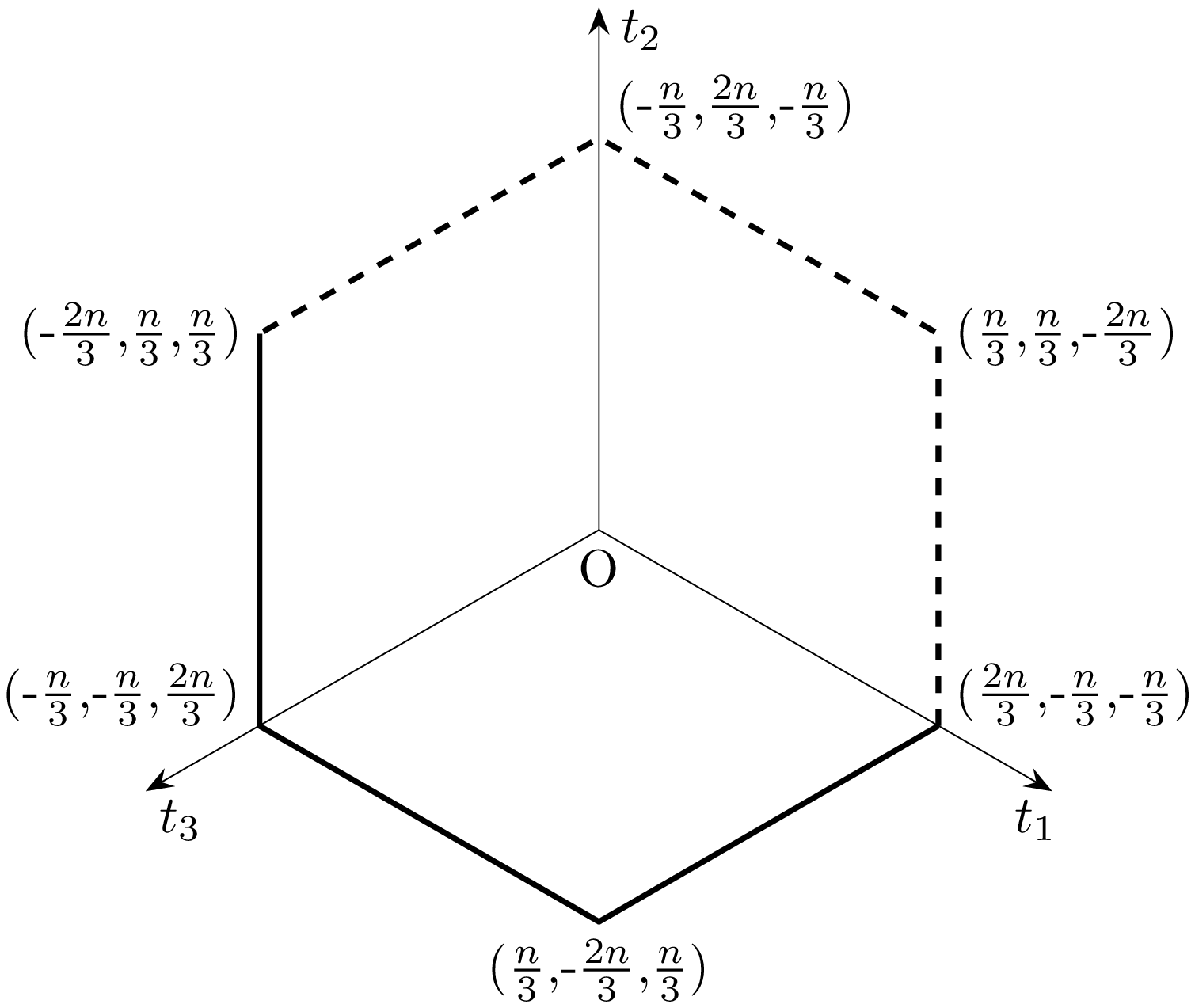}
\hspace*{\fill}
\caption{ The hexagon domain and $\Omega_B$}
\end{figure}
Here the index set $\Lambda_N^\dag = \left\{k\in \ZZ^2: \ -n\le 2k_2+k_1, k_2-k_1, 2k_1+k_2<n
\right\}$, which becomes in homogeneous coordinates $\KK_n^\dag$ defined by
\begin{align*}
  \KK_n^{\dag} : = \{\jb \in \HH: -n \le j_2-j_1,j_1-j_3, j_2- j_3 < n \}.
\end{align*}
We also have $\Lambda_N = \left\{k\in \ZZ^2: \ -n\le 2k_2-k_1, k_1+k_2, 2k_1-k_2<n \right\}$.
Recall that $\tb \equiv 0 \pmod 3$ means, by definition, $t_1\equiv t_2 \equiv t_3
 \pmod 3$. It is not hard to see that the set
$\Lambda_N $ becomes, in homogeneous coordinates, $\KK_n$ defined by
\begin{align*}
     \KK_n:  = \{\jb \in \HH: -n \le j_1,j_2,-j_3<n, \   \jb \equiv 0 \pmod 3 \}.
\end{align*}
We also denote by $\KK_n^{\dag*}$ and $\KK_n^*$ the sets defined with
$\le$ in place of $<$ in $\KK_n^\dag$ and $\KK_n$, respectively. The set $\KK_n^*$
can be obtained form a rotation of $\KK_n^{\dag*}$, as shown in the following
proposition, which can be easily verified.

\begin{prop} \label{prop:H*H}
For  $\kb=(k_1,k_2,k_3)\in \HH$, define $\wh{\kb}:=(k_3-k_2,k_1-k_3,k_2-k_1)$.
Then $\frac{\wh\kb}{3}\in \KK_n^{\dag*}$ if $\kb\in \KK_n^*$ and
 $\wh\kb\in \KK_n^{*}$ if $\kb\in \KK_n^{\dag*}$.
\end{prop}

The finite dimensional space $\CH_n$ of exponential functions becomes
$$
  \CK_n : =   \sspan \left \{ \phi_\jb(\tb) = \e^{ \frac{2 \pi i}{3} \jb^\tr \tb}: \jb \in
    \KK_n^{\dag} \right \} \quad \hbox{and} \quad
  \CK_n^*: =   \sspan \left \{ \phi_\jb: \jb \in  \KK_n^{\dag*} \right \}.
$$
By induction, it follows that $\dim \CK_n = |\KK_n| = n^2$ and $\dim \CK_n^* = |\KK_n^*|
 = n^2 + n +1$ if $n=0,2 \pmod 3$ and  $|\KK_n^*| = n^2 + n -1$ if $n=1 \pmod 3$.
\begin{figure}[h]
\begin{minipage}{0.4\textwidth}\centering \includegraphics[width=1\textwidth]{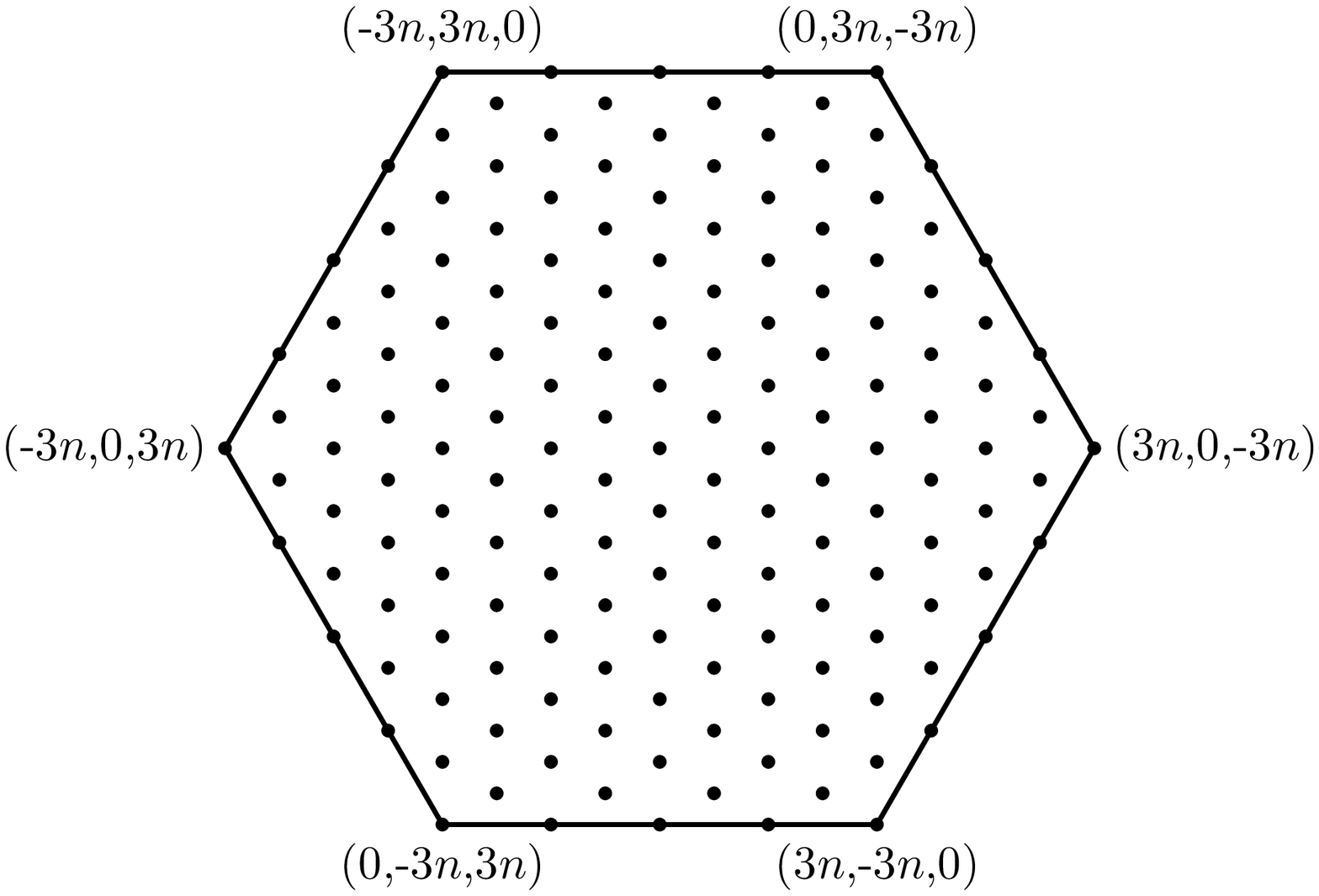}
\end{minipage}
\begin{minipage}{0.35\textwidth}\centering \includegraphics[width=1\textwidth]{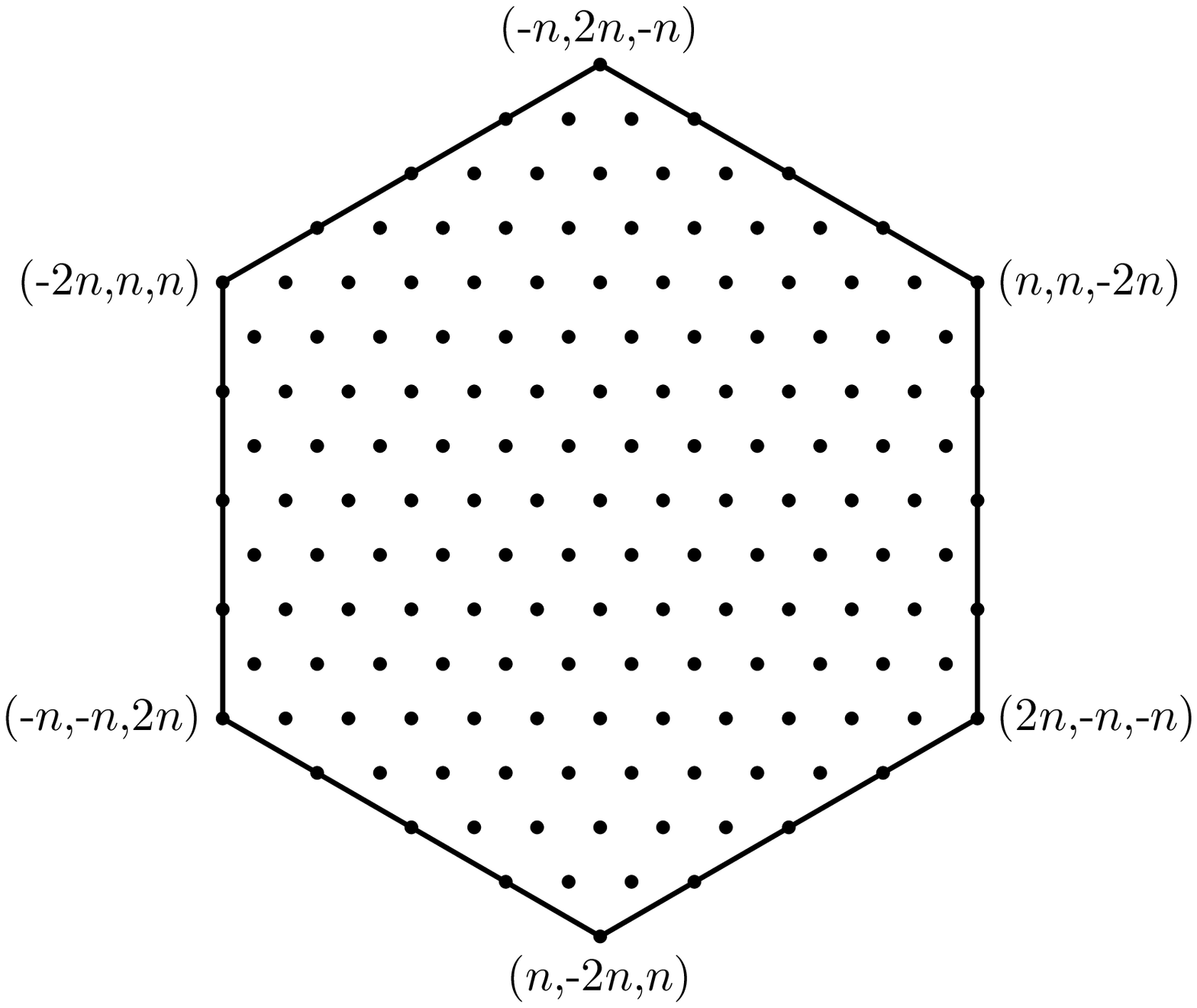}
\end{minipage}\caption{ The set $\KK^*_{3n}$  (left) and the set
   $\KK^{\dag*}_{3n}$ (right).}
\label{K*K}
 \end{figure}
The two sets $\KK_{3n}^*$ and $\KK_{3n}^{\dag*}$ take different shapes, which we
depict in Figure \ref{K*K}. Define
\begin{equation} \label{ipHH2}
\la f,g\ra_n:= \frac{1}{n^2} \sum_{\jb \in \KK_n} f(\tfrac{\jb}{n})\overline{f(\tfrac{\jb}{n})}
  \quad \hbox{and}\quad
\la f,g\ra_n^*=\frac{1}{n^2} \sum_{\jb \in \KK_n^*} c_{\jb}^{(n)}f(\tfrac{\jb}{n})
    \overline{f(\tfrac{\jb}{n})},
\end{equation}
where $c_{\jb}^{(n)}$ are as defined in \eqref{cuba-HH} with $\HH_n$ replaced by $\KK_n$.

\begin{thm}\label{thm:HH2ip}
For $n \ge 0$,
$ \frac{1}{|\Omega|}\int_\Omega f(x) \overline{g(x)} dx  = \la f, g\ra_n = \la f, g\ra_n^*$
for $f, g \in \CK_n$. In particular, $\la \phi_\jb, \phi_\kb\ra_n^* = 1$ if $\hat \jb \equiv \hat \kb \pmod{3n}$ and $\la \phi_\jb, \phi_\kb\ra_n^* = 0$ otherwise, for $\kb, \jb \in \KK_n$.
Moreover, we have the cubature
\begin{align}  \label{cuba-HHD}
   \frac{1}{|\Omega|} \int_{\Omega} f(\tb)d\tb  = \frac{1}{n^2} \sum_{\jb\in \CK_n^*} c_\jb^{(n)}
   f(\tfrac{\jb}{n}), \quad \forall f\in \CK_{2n-1}^*.
\end{align}
\end{thm}

The part of the theorem on $\la f, g\ra_n $ is exactly Theorem \ref{thm:df}, while the part
on $\la f, g\ra_n^*$ and the cubature can be proved by periodicity, just like the proof of
Theorem \ref{thm:HH1} in \cite{LSX1}, upon using the Proposition \ref{prop:H*H}.
The cubature \eqref{cuba-HHD} is already one in Stage 2; we can also derive a cubature
with nodes indexed by $\KK_n$ as in Stage 1.

Next we consider the invariant and anti-invariant functions under $\CA_2$, which are
the generalized cosines $\TC_\kb$ and the generalized sines $\TS_\kb$ considered
in the previous subsection. By restricting to such functions, we again obtain cubature
on the triangle $\Delta$. The index set of the nodes of the cubature, denoted by
$\Upsilon$, is
\begin{align*}
 \Upsilon _n := \{\jb \in \HH: 0 \le j_1,j_2, -j_3 \le n, \jb \equiv 0 \pmod{3}\}
\end{align*}
derived by symmetry from $\KK_n^*$, whereas the index set of the invariant functions
being integrated exactly by the cubature, denoted by $\Upsilon^\dag$, is derived
from $\KK_n^{\dag*}$,
\begin{align*}
 \Upsilon _n^\dag  = \{\jb \in \HH: 0 \le j_1,j_2, -j_3 \le n,j_2 - j_3 \le n, j_1-j_3 \le n\},
\end{align*}
which is inside a quadrilateral;
Figure \ref{QuadF} shows its relative position in $\Omega_B$.
\begin{figure}[h]
\centering
\begin{minipage}{0.4\textwidth}\centering \includegraphics[width=1\textwidth]{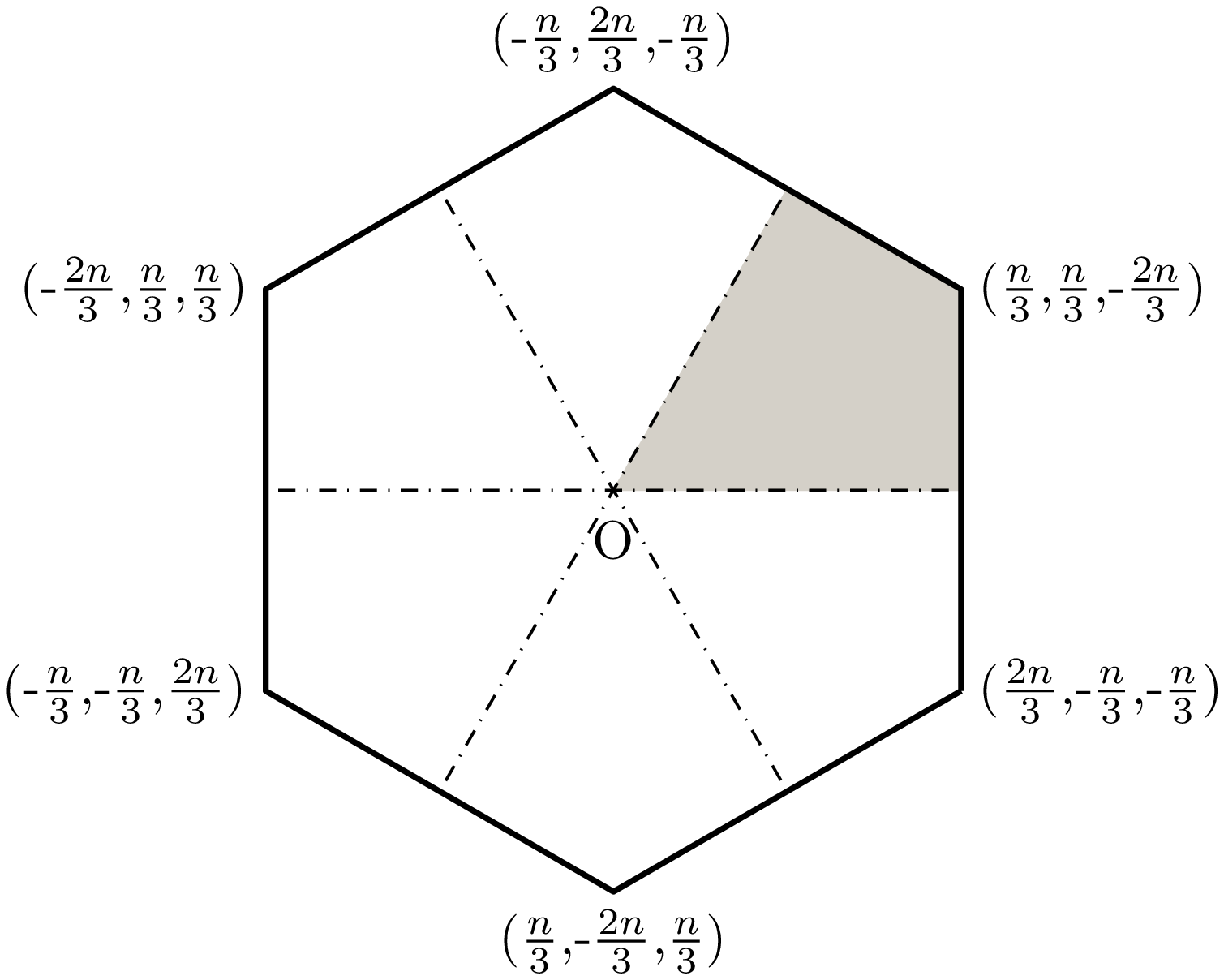}
\end{minipage}\qquad
\begin{minipage}{0.35\textwidth}\centering \includegraphics[width=1\textwidth]{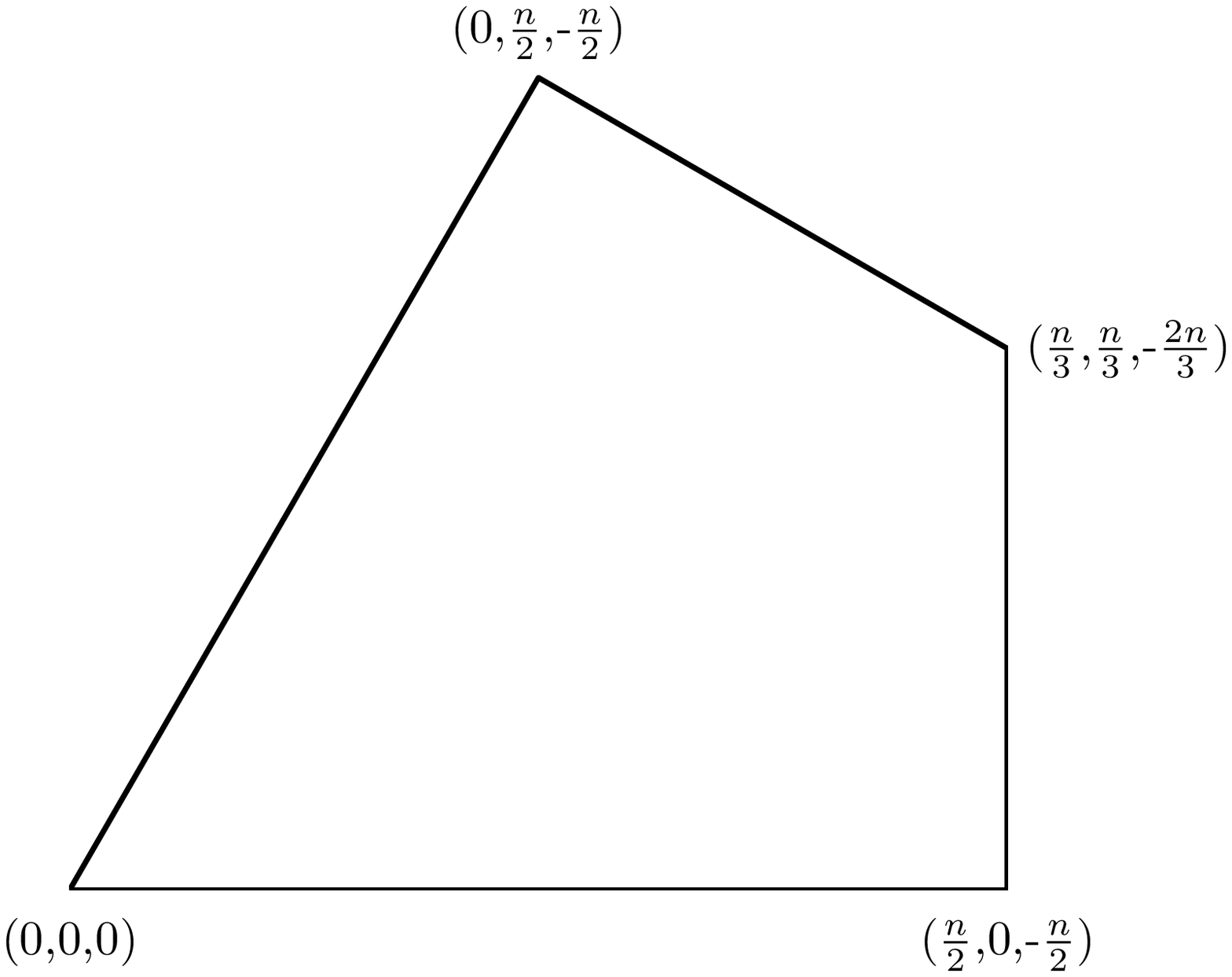}
\end{minipage}
\caption{The fundamental quadrilateral of $\Omega_B$ under $\A_2$}
\label{QuadF}
\end{figure}
We define the following subspaces of trigonometric functions,
$$
  \CTC_n = \sspan \{\TC_\kb: \kb \in \Upsilon_n^\dag\} \quad \hbox{and} \quad
  \CTS_n = \sspan \{\TS_\kb: \kb \in \Upsilon_n^{\dag \circ}\}.
$$
The set $\Upsilon_n$ takes a symmetric form when $n$ is a multiple of $3$. In
Figure \eqref{FourM} we depict the index sets $\Upsilon_{3n}$ and $\Upsilon_{3n}^\dag$.
\begin{figure}[h]
\centering
\begin{minipage}{0.3\textwidth}\centering \includegraphics[width=1\textwidth]{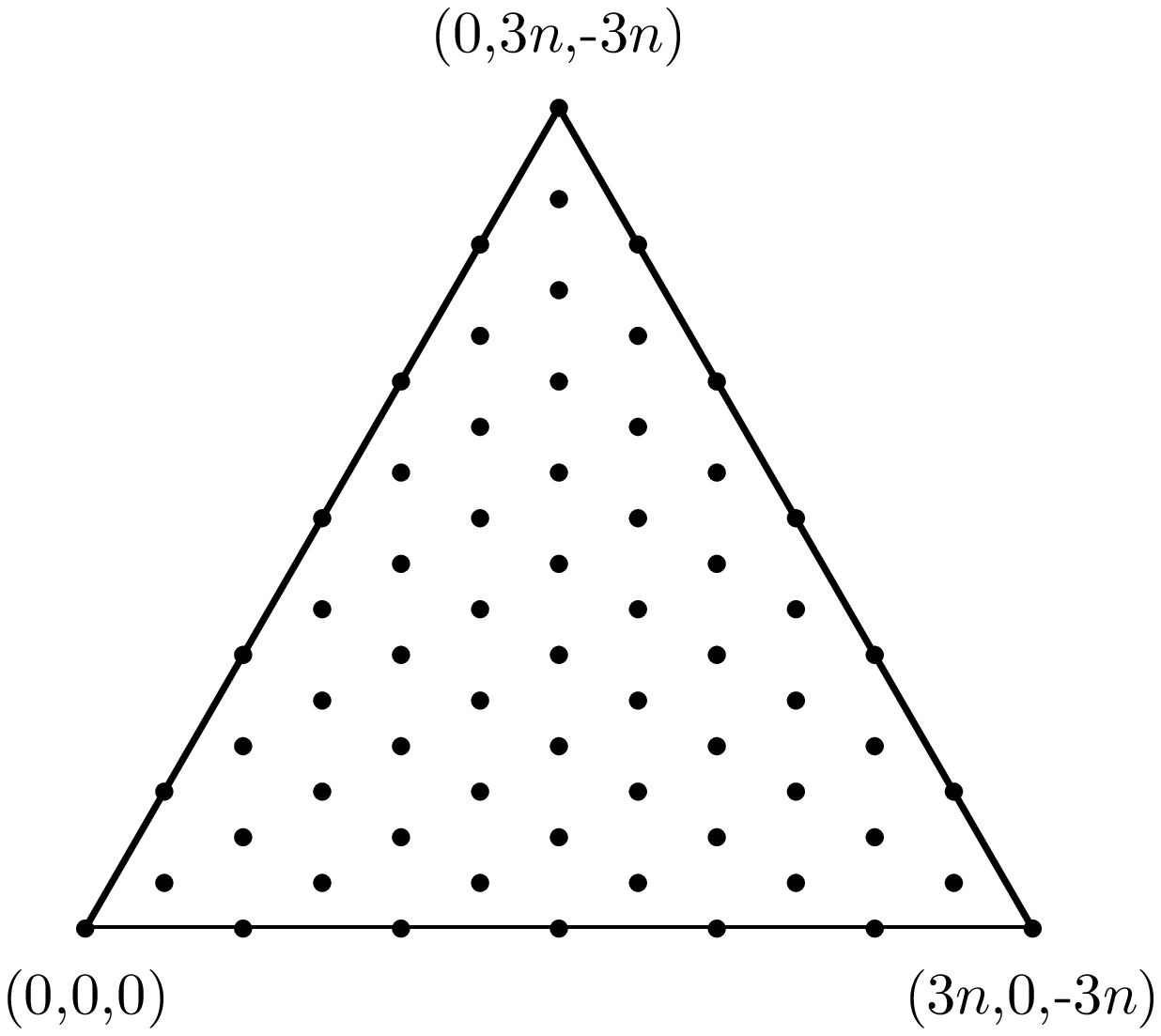}
\end{minipage}
\qquad
\begin{minipage}{0.31\textwidth}\centering \includegraphics[width=1\textwidth]{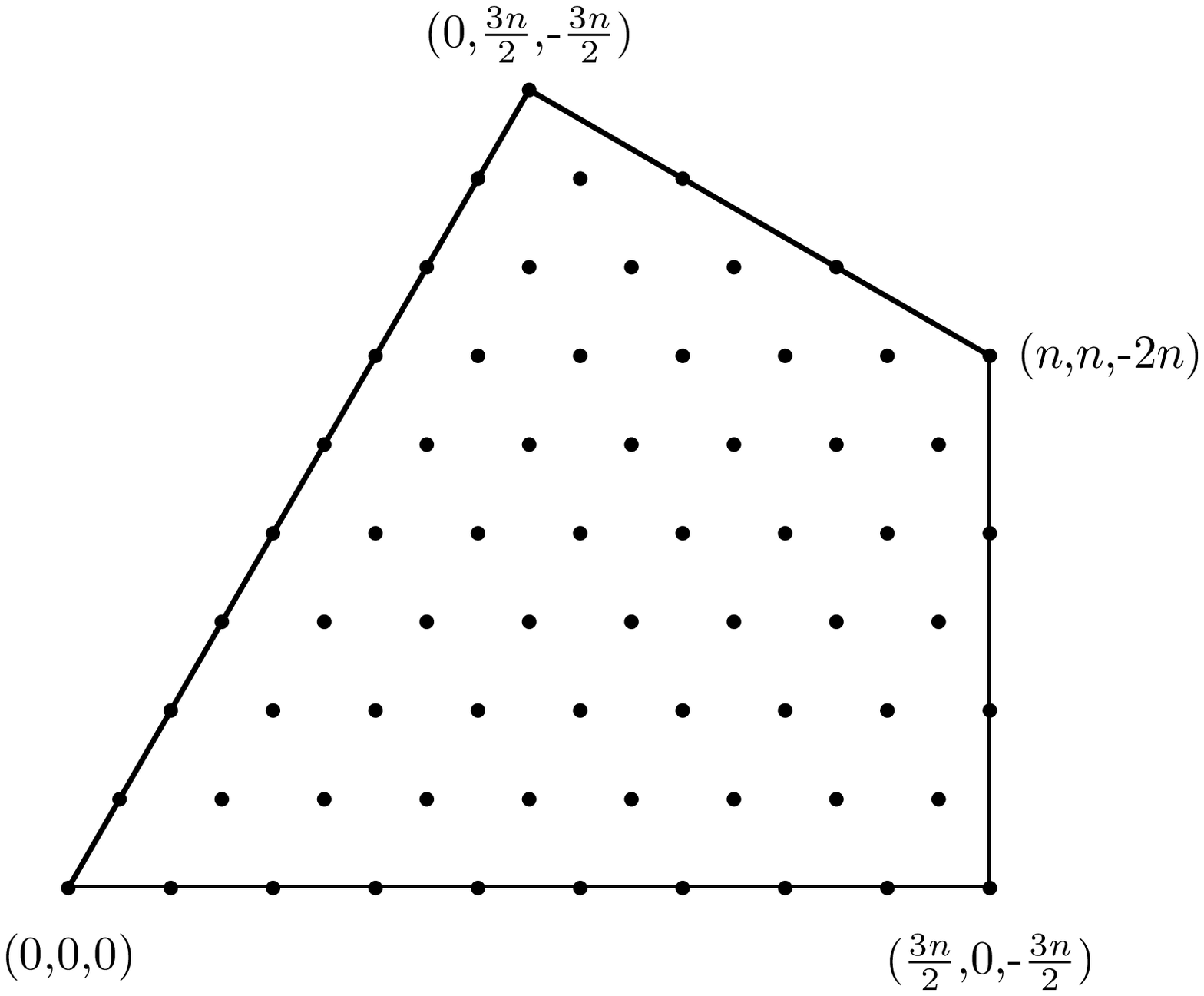}
\end{minipage}
\caption{the index sets $\Upsilon_{3n}$ (left) and $\Upsilon_{3n}^{\dag}$ (right).}
\label{FourM}
\end{figure}

\begin{thm} \label{thm:inner-triangle}
The following cubature is exact for all $f\in \CTC_{2n-1}$,
\begin{align} \label{cuba-HHT2}
\frac{1}{|\Delta|}\int_{\Delta}f(\tb) d\tb = \frac{1}{n^2} \sum_{\jb \in \Upsilon_n}
        \lambda_{\jb}^{(n)} f(\tfrac{\jb}{n}), \qquad \lambda_\jb^{(n)} : =  \begin{cases}
       6, & \jb \in \Upsilon^\circ_n,\\
       3 , & \jb \in \Upsilon^\e_n,\\
       1, & \jb \in \Upsilon^\ve_n.
       \end{cases}
\end{align}
\end{thm}

The formula \eqref{cuba-HHT2} is derived from \eqref{cuba-HHD} by using the invariance
of the functions in $\CTC_{2n-1}$ and the fact
$\Omega = \big(\cup_{\sigma\in \A_2} \{ \tb\sigma :  \tb \in \Delta^{\circ}\} \big) \cup \big( \cup_{\sigma\in \A_2} \{ \tb\sigma :  \tb \in \partial\Delta \} \big)$. As the proof
is similar to that of \eqref{cubaHH2} in \cite{LSX1}, we shall omit the details.

Similarly, we can also derive a cubature for $\CTS_{2n-1}$ based on points in
$\Upsilon_n^\circ$. These are cubature in Stage 3. We note that the set of nodes
in \eqref{cuba-HHT2} is different from that in \eqref{cubaHH2}, see
Figure 2 and Figure 8, even though both are on the triangle.

As in the case of Hexagon-Hexagon, we can continue to Stage 4, where the cubature
\eqref{cuba-HHT2} is mapped by the change of variables \eqref{x-y} to an algebraic
cubature for $w_{-1/2}(x,y) dxdy$ on $\Delta^*$, the region bounded by SteinerÕs
hypocycloid, which is a cubature exact for all polynomials in $\Pi_{2n-1}^2$ but with
many more nodes than the one derived in the Hexagon-Hexagon case.

The set of cubature points in \eqref{cuba-HHT2} and its image in the region bounded by
Steiner's hypocycloid have also been studied in \cite{Mun}.

\subsubsection{Interpolation} Applying Theorem \ref{prop:interpolation} to the current
set up, we obtain an interpolation operator $\CI_n f$ that interpolates $f$ on $\KK_n$
in the hexagon. We would like to consider interpolation on the triangle based on points
in $\Upsilon_n$. For this purpose, we first construct a near interpolation operator on
the symmetric set of points $\KK_n^*$.

\begin{thm} \label{prop:tH-interpo}
Let $\S_\jb: = \left\{\kb \in \KK_n^*:  \kb  \equiv  \jb  \pmod {3n} \right\}.$
For $f \in C(\Omega)$, define
$$
 \CI_n^* f (\tb): = \sum_{\jb \in \KK_n^*}f(\tfrac{\jb}{n}) \Phi_n(\tb - \tfrac{\jb}{n}), \quad
     \Phi_n(\tb) := \frac{1}{n^2} \sum_{\jb \in \KK_n^{\dag*}} c_{\wh \jb}^{(n)}
         \phi_\jb(\tb).
$$
Then $\CI_n^* f \in \CK_n^*$ and $\CI_n^* f (\tfrac{\jb}{n}) = f(\tfrac{\jb}{n})$ if
$\jb \in \KK_n^\circ$, and $\CI_n^* f (\tfrac{\jb}{n}) = \sum_{\kb \in \S_\jb} f(\frac{\kb}{n})$
if $\jb \in \partial \KK_n^*$, the boundary of $\KK_n^*$. Furthermore, $\Phi_n^*(\tb)$ is a
real function and it is given by the following formula when $n=0 \pmod 3$,
\begin{align} \label{tPhi_n}
  \Phi_n(\tb) = & \frac{1}{n^2} \left[ -\frac{1}{2} \sum_{i=1}^3
    \frac{(2\cos \pi s_i+\cos \pi t_i) \sin \pi t_i \cos \frac{2\pi n t_i }{3} }
     {\sin \pi t_1 \sin \pi t_2 \sin \pi t_3}  \right. \\
  & \qquad
      \left.  - \frac{1}{3} \big( \cos \tfrac{2n\pi t_1 }{3}+ \cos \tfrac{2n\pi t_2}{3}
                + \cos  \tfrac{2n\pi s_3 }{3} \big) \right],    \notag
\end{align}
where $s_1=\frac{t_3-t_2}{3}, s_2=\frac{t_1-t_3}{3}, s_3=\frac{t_2-t_1}{3}$.
\end{thm}

\begin{proof}
By Proposition \ref{prop:H*H}, $\kb\in \KK_n^{*}$ implies $\frac{\wh \kb}{3} \in \KK_n^{\dag*}$,
and $\jb\in \KK_n^{\dag*}$ implies $\wh{\jb} \in \KK_n^{*}$.  By homogeneity,
$\kb\cdot \tb=(k_3-k_2)s_1+(k_1-k_3)s_2+(k_2-k_1)s_3 = \wh \kb \cdot \sb$. As a result,
\begin{align*}
\Phi_n(\tb) =
\frac{1}{n^2} \sum_{\jb \in \KK_n^{\dag*}} {c}_{\wh \jb}^{(n)} \phi_{\wh \jb}(\sb)
 =\frac{1}{n^2} \sum_{\jb \in \KK_n^{*}} {c}_{\jb}^{(n)} \phi_{\jb}(\sb).
\end{align*}
Consequently, by the definition in \eqref{ipHH2},
\begin{align*}
 \Phi_n (\tfrac{\kb-\jb}{n}) =
    \frac{1}{n^2} \sum_{\lb \in \KK_n^*} c_\lb^{(n)} \phi_\lb(\tfrac{\wh{\kb}-\wh{\jb}}{n})
=\frac{1}{n^2} \sum_{\lb \in \KK_n^*} c_\lb^{(n)} \phi_{\wh{\kb}-\wh{\jb}}(\tfrac{\lb}{n})
 = \la \phi_{\wh{\kb}},\phi_{\wh{\jb}} \ra^*_n.
\end{align*}
Thus, by Theorem \ref{thm:HH2ip}, it follows that
\begin{align} \label{k=jmod}
\begin{split}
   \Phi_n (\tfrac{\kb-\jb}{n}) = \la \phi_{\wh\kb}, \phi_{\wh\jb}\ra_{n}^* =
\begin{cases}
  1, & \wh\kb = \wh\jb + n \lb,\, \lb\in \HH,\\ 0, & \text{otherwise},
\end{cases}
\end{split}
\end{align}
which proves the stated result of $\CI_n^*f$.

To derive the compact formula for $\Phi_n$ we essentially need a formula for the
Dirichlet kernel, denoted by $\Theta_n(\tb)$, of the Fourier series over $\HH^{\dag*}_n$,
\begin{align*}
\Theta_n(\tb) = \sum_{\jb\in \HH^{\dag*}_n}  \phi_{\jb}(\tb).
\end{align*}
Indeed, by the definition of $c_\jb^{(n)}$, it follows that
\begin{align} \label{wPhi}
n^2 {\Phi}_n(\tb) = \frac{1}{2} \big(  {\Theta}_n(\tb) +  {\Theta}_{n-1}(\tb)\big)
- \begin{cases} \TC_{\frac{n}{3},\frac{n}{3},-\frac{2n}{3}}(\tb), & n \equiv 0 \pmod{3},\\
0, & otherwise.
\end{cases}
\end{align}

Using the identity $\kb\cdot \tb= \wh \kb \cdot \sb$ and Proposition \ref{prop:H*H}, we derive that
\begin{align*}
{\Theta}_n(\tb) = \sum_{\kb\in \KK_n^{\dag*}} \phi_\kb(\tb) =
    \sum_{\kb\in \KK_n^{\dag*}}\phi_{\wh \kb}(\sb) = \sum_{\jb\in \KK_n^{*}}\phi_{\jb}(\sb).
\end{align*}
We now partition $\KK_n^{*}$ into three parts according to the
congruence relation,
\begin{align*}
\KK_n^{(0)}:=\big\{ \jb \in {\KK}_n^{*}: j_1\equiv j_2\equiv j_3\equiv 0 \pmod 3\big\},\\
\KK_n^{(1)}:=\big\{ \jb \in {\KK}_n^{*}: j_1\equiv j_2\equiv j_3\equiv 1 \pmod 3\big\},\\
\KK_n^{(2)}:=\big\{ \jb \in {\KK}_n^{*}: j_1\equiv j_2\equiv j_3\equiv 2 \pmod 3\big\}.
\end{align*}
Using the fact that $\jb\in\KK_n^{(0)}  \Longleftrightarrow \frac{\jb}{3}\in \HH_{\lfloor\frac{n}{3}\rfloor}^*$, where $\HH_n^*$ is the
index defined in the previous subsection, and $\phi_{\jb}(\sb) = \phi_{\frac{\jb}{3}}(3\sb)$
if $\jb\in\KK_n^{(0)}$, we obtain from the Dirichlet kernel over $\HH_n^*$ in (3.10) of \cite{LSX1},
\begin{align} \label{SUM0}
\begin{split}
 \sum_{\jb\in \KK_n^{(0)}}\phi_{\jb}(\sb)
& = \prod_{j=1}^3 \frac{\sin \pi \lfloor\frac{n+3}{3}\rfloor t_j}{\sin \pi t_j}
    -\prod_{j=1}^3 \frac{\sin \pi \lfloor\frac{n}{3}\rfloor t_j}{\sin \pi t_j} \\
& = \sum_{j=1}^3 \frac{ \sin 2\pi \lfloor\frac{n}{3}\rfloor t_j - \sin 2\pi \lfloor\frac{n+3}{3}\rfloor t_j}
   {4\sin \pi t_1 \sin \pi t_2 \sin \pi t_3},
\end{split}
\end{align}
where we have used the identity $\sin 2t_1+\sin 2t_2+\sin2t_3=-4\sin t_1\sin t_2 \sin t_3$
in the last equal sign \cite[(3.15)]{LSX1}. Next we note that $\KK_n^{(1)}$ can be divided
into the following three (non-overlapping) subsets
$\KK_n^{(1)} = \KK_n^{(1,3)} \cup \KK_n^{(1,2)} \cup \KK_n^{(1,1)}$, where
\begin{align*}
&\KK_n^{(1,1)} = \left\{(-j_2-j_3,j_2,j_3):  j_2\equiv j_3 \equiv 1 \pmod3,\    1\le j_2,-j_3 \le n\right\},\\
&\KK_n^{(1,2)} = \left\{(j_1,-j_1-j_3,j_3): j_3\equiv j_1 \equiv 1 \pmod3,\    1\le j_3,-j_1 \le n\right\}, \\
&\KK_n^{(1,3)} = \left\{(j_1,j_2,-j_1-j_2):  j_1\equiv j_2 \equiv 1 \pmod3,\    1\le j_1,-j_2 \le n\right\}.
\end{align*}
Using the last set $\KK_n^{(1,3)}$, we define
\begin{align} \label{CI}
\begin{split}
\CI(t_1,t_2) :=&\ \sum_{\jb\in \KK_n^{(1,3)}}\phi_{\jb}(\sb) = \sum_{1\le j_1\le n \atop 3 | j_1-1} \sum_{1\le -j_2\le n \atop 3 |j_2-1} \e^{\frac{2i\pi}{3}(j_2t_1-j_1t_2 ) }
\\ =&\ \frac{\e^{-\frac{2i\pi t_2}{3}} (1-\e^{-2i\pi \lfloor \frac{n+2}{3}\rfloor t_2} )  }{1-\e^{-2i\pi t_2}}
\frac{\e^{-\frac{4i\pi t_1}{3}} (1-\e^{-2i\pi \lfloor \frac{n+1}{3}\rfloor t_1} )  }{1-\e^{-2i\pi t_1}}
\\ =&\ \frac{(\e^{-2i\pi s_2} -\e^{-2i\pi s_1}) (1-\e^{-2i\pi \lfloor \frac{n+1}{3}\rfloor t_1} )  (1-\e^{-2i\pi \lfloor \frac{n+2}{3}\rfloor t_2} )  }{(1-\e^{-2i\pi t_1}) (1-\e^{-2i\pi t_2}) (1-\e^{-2i\pi t_3}) },
\end{split}
\end{align}
where the second equal sign follows from $\jb \cdot \sb = j_1(s_1-s_3)+j_2(s_2-s_3)
=j_2t_1 -j_1t_2$. Moreover, we have $ \jb \cdot \sb = j_2t_1 -j_1t_2=j_3t_2 -j_2t_3 =
 j_1t_3 -j_3t_1$, which yields
\begin{align*}
\CI(t_2,t_3)=
\sum_{\jb\in \KK_n^{(1,1)}}\phi_{\jb}(\sb) \quad \hbox{and}\quad \CI(t_3,t_1)=
\sum_{\jb\in \KK_n^{(1,2)}}\phi_{\jb}(\sb).
\end{align*}
As a result, we conclude
\begin{align*}
  \sum_{\jb\in \KK_n^{(1)}}\phi_{\jb}(\sb) =  \CI(t_1,t_2)+\CI(t_2,t_3)+\CI(t_3,t_1).
\end{align*}
Furthermore, we note that $\KK_n^{(2)} =\left\{-\jb: \jb\in  \KK_n^{(1)} \right\}$ and, consequently,
\begin{align*}
\sum_{\jb\in \KK_n^{(1)}\cup \KK_n^{(2)}}\phi_{\jb}(\sb)& =
\sum_{\jb\in \KK_n^{(1)}}\phi_{\jb}(\sb) + \sum_{\jb\in \KK_n^{(2)}}\phi_{\jb}(\sb)
 = \sum_{\jb\in \KK_n^{(1)}}\phi_{\jb}(\sb) +
 \sum_{\jb\in \KK_n^{(1)}}\phi_{\jb}(-\sb) \\
& = \sum_{\jb\in \KK_n^{(1)}}\phi_{\jb}(\sb)+\sum_{\jb\in \KK_n^{(1)}} \overline{\phi_{\jb}(\sb)}
= 2\Re \{\CI(t_1,t_2) + \CI(t_2,t_3) + \CI(t_3,t_1) \}.
\end{align*}
Now assume that $n$ is a multiple of $3$. By using \eqref{CI} and the fact that $\tb$ is
homogeneous, we obtain
\begin{align*}
\begin{split}
\sum_{\jb\in \KK_n^{(1)}\cup \KK_n^{(2)} }   \phi_{\jb}(\sb)
=   2\Re &  \{ \CI(t_1,t_2))+2\Re(\CI(t_2,t_3))+2\Re(\CI(t_3,t_1) \} \\
=   2\Re & \left \{
 \frac{(\e^{-2i\pi s_2} -\e^{-2i\pi s_1}) (1+ \e^{2i\pi\frac{n}{3}t_3} -\e^{-2i\pi\frac{n}{3}t_1}-\e^{-2i\pi\frac{n}{3}t_2} ) }{(1-\e^{-2i\pi t_1}) (1-\e^{-2i\pi t_2}) (1-\e^{-2i\pi t_3}) }  \right.
\\ &+  \frac{(\e^{-2i\pi s_3} -\e^{-2i\pi s_2}) (1+ \e^{2i\pi\frac{n}{3}t_1} -\e^{-2i\pi\frac{n}{3}t_2}-\e^{-2i\pi\frac{n}{3}t_3} ) }{(1-\e^{-2i\pi t_1}) (1-\e^{-2i\pi t_2}) (1-\e^{-2i\pi t_3}) } \\
 & \left. + \frac{(\e^{-2i\pi s_1} -\e^{-2i\pi s_3}) (1+ \e^{2i\pi\frac{n}{3}t_2} -\e^{-2i\pi\frac{n}{3}t_3}-\e^{-2i\pi\frac{n}{3}t_1} ) }{(1-\e^{-2i\pi t_1}) (1-\e^{-2i\pi t_2}) (1-\e^{-2i\pi t_3}) } \right\}.
\end{split}
\end{align*}
Combining the numerators and collecting the terms in $1$, $\e^{2i\pi\frac{n}{3}\cdot}$
and $\e^{-2i\pi\frac{n}{3}\cdot}$, we obtain that the combined numerator is equal to
\begin{align*}
\begin{split}
 &\ (\e^{-2i\pi s_2} -\e^{-2i\pi s_1}) + (\e^{-2i\pi s_3} -\e^{-2i\pi s_2}) + (\e^{-2i\pi s_1} -\e^{-2i\pi s_3})
\\ &+ (\e^{-2i\pi s_2} -\e^{-2i\pi s_1}) \e^{2i\pi\frac{n}{3}t_3} -\big( (\e^{-2i\pi s_3} -\e^{-2i\pi s_2})+(\e^{-2i\pi s_1} -\e^{-2i\pi s_3})\big)\e^{-2i\pi\frac{n}{3}t_3}
\\ &+ (\e^{-2i\pi s_3} -\e^{-2i\pi s_2}) \e^{2i\pi\frac{n}{3}t_1} -\big( (\e^{-2i\pi s_1} -\e^{-2i\pi s_3})+(\e^{-2i\pi s_2} -\e^{-2i\pi s_1})\big)\e^{-2i\pi\frac{n}{3}t_1}
\\ &+ (\e^{-2i\pi s_1} -\e^{-2i\pi s_3}) \e^{2i\pi\frac{n}{3}t_2} -\big( (\e^{-2i\pi s_2} -\e^{-2i\pi s_1})+(\e^{-2i\pi s_3} -\e^{-2i\pi s_2})\big)\e^{-2i\pi\frac{n}{3}t_2}
\\
=&\  (\e^{-2i\pi s_2} -\e^{-2i\pi s_1}) (\e^{2i\pi\frac{n}{3}t_3}+\e^{-2i\pi\frac{n}{3}t_3})
    + (\e^{-2i\pi s_3} -\e^{-2i\pi s_2}) (\e^{2i\pi\frac{n}{3}t_1} + \e^{2i\pi\frac{n}{3}t_1} )
\\ &+ (\e^{-2i\pi s_1} -\e^{-2i\pi s_3}) (\e^{2i\pi\frac{n}{3}t_2}+\e^{-2i\pi\frac{n}{3}t_2})
\\ =&\  2  \cos \tfrac{2\pi n t_3}{3} \e^{i\pi s_3} (\e^{i\pi t_3} -\e^{-i\pi t_3})  +  2  \cos \tfrac{2\pi n t_1}{3} \e^{i\pi s_1} (\e^{i\pi t_1} -\e^{-i\pi t_1})
\\ &+  2  \cos \tfrac{2\pi n t_2}{3} \e^{i\pi s_2} (\e^{i\pi t_2} -\e^{-i\pi t_2})
\\ =&\  4i  \e^{i\pi s_3} \cos \tfrac{2\pi n t_3}{3}  \sin \pi t_3  + 4i\e^{i\pi s_1}   \cos \tfrac{2\pi n t_1}{3} \sin \pi t_1
+  4i\e^{i\pi s_2}  \cos \tfrac{2\pi n t_2}{3} \sin \pi t_2,
\end{split}
\end{align*}
where we use the facts that $ t_3=s_1 - s_2 , t_1=s_2-s_3, t_2=s_3-s_1$ and $s_1+s_2+s_3=0$  for the second equal sign. Using $t_1+t_2+t_3 =0$, the denominator becomes
\begin{align*}
(1-&\e^{-2i\pi t_1}) (1-\e^{-2i\pi t_2}) (1-\e^{-2i\pi t_3})
 =-8i \sin \pi t_1  \sin \pi t_2 \sin \pi t_3.
\end{align*}
Consequently, we derive that
\begin{align*} 
\begin{split}
& \sum_{\jb\in \KK_n^{(1)}\cup \KK_n^{(2)}}   \phi_{\jb}(\sb) \\
  & \qquad =   2\Re \frac{\e^{i\pi s_3} \cos \tfrac{2\pi n t_3}{3}  \sin \pi t_3  + \e^{i\pi s_1}   \cos \tfrac{2\pi n t_1}{3} \sin \pi t_1
+  \e^{i\pi s_2}  \cos \tfrac{2\pi n t_2}{3} \sin \pi t_2}{-2 \sin \pi t_1  \sin \pi t_2 \sin \pi t_3} \\
  &  \qquad = -\sum_{j=1}^3
\frac{\cos \frac{2\pi n t_j}{3} \cos \pi s_j \sin \pi t_j}{ \sin \pi t_1 \sin \pi t_2  \sin \pi t_3}.
\end{split}
\end{align*}
Combining the above equation with \eqref{wPhi} and \eqref{SUM0}, we obtain
\begin{align*}
& n^2 \Phi_n (\tb) = \frac{1}{2} \big( {\Theta}_n(\tb) + {\Theta}_{n-1}(\tb)\big)
- \TC_{\frac{n}{3},\frac{n}{3},-\frac{2n}{3}}(\tb)
\\ =&\ \frac{1}{2}\bigg(\sum_{\jb\in \KK_n^{(1)} } \phi_{\jb}(\sb) +
 \sum_{\jb\in \KK_n^{(1)}\cup \KK_n^{(2)}}\phi_{\jb}(\sb)+
 \sum_{\jb\in \KK_{n-3}^{(1)} } \phi_{\jb}(\sb)
 \bigg)
- \TC_{\frac{n}{3},\frac{n}{3},-\frac{2n}{3}}(\tb)
\\ =& \sum_{j=1}^3 \frac{ \sin \frac{2\pi (n-3) t_j}{3}- \sin \frac{2\pi (n+3) t_j}{3}
}{8\sin \pi t_1 \sin \pi t_2 \sin \pi t_3} -\sum_{j=1}^3
\frac{\cos \frac{2\pi n t_j}{3} \cos \pi s_j \sin \pi t_j}{ \sin \pi t_1 \sin \pi t_2  \sin \pi t_3}
- \TC_{\frac{n}{3},\frac{n}{3},-\frac{2n}{3}}(\tb)
\\ =& -\sum_{j=1}^3 \frac{ \cos \frac{2\pi n t_j}{3} \sin \pi t_j \cos \pi t_j
}{2\sin \pi t_1 \sin \pi t_2 \sin \pi t_3} -\sum_{j=1}^3
\frac{\cos \frac{2\pi n t_j}{3} \cos \pi s_j \sin \pi t_j}{ \sin \pi t_1 \sin \pi t_2  \sin \pi t_3}
- \TC_{\frac{n}{3},\frac{n}{3},-\frac{2n}{3}}(\tb)
\\ =& -\sum_{j=1}^3 \frac{ \cos \frac{2\pi n t_j}{3} \sin \pi t_j (\cos \pi t_j+2\cos\pi s_j)
}{2\sin \pi t_1 \sin \pi t_2 \sin \pi t_3}
- \frac{1}{3} \big(\cos \tfrac{2\pi n t_1}{3} + \cos \tfrac{2\pi n t_2}{3}+ \cos \tfrac{2\pi n t_3}{3} \big).
\end{align*}
This completes the proof.
\end{proof}

We now proceed to interpolation on the triangle $\Delta$. The idea is to use the
periodicity and apply the operator $\CP^\pm$ in \eqref{CP-pm} on the interpolation
$\CI_n f$, as in Theorem 4.7 in \cite{LSX1}. First we apply $\CP^-$ on $\CI_n f$, which
gives the following:

\begin{thm} \label{prop:1st-interpo-Dela}
For $n\ge 0$ and $f\in C(\Delta)$ define
$$
 \CL_n f (\tb) : = \sum_{\jb \in \Upsilon_n^{\circ}} f(\tfrac{\jb}{n})
      \ell_{\jb,n}^{\circ} (\tb), \qquad
       {\ell}_{\jb,n}^\circ (\tb) = \frac{6}{n^2} \sum_{\kb \in \Upsilon_n^{\dag\circ}}
        \wh{\lambda}_{\kb}^{(n)} \TS_{\kb}(\tb)\overline{ \TS_{\kb}(\tfrac{\jb}{n}) },
$$
where
$$
\wh{\lambda}_{\kb}^{(n)} =  {c}_{\wh{\kb}}^{(3n)} |\kb \A_2| = \begin{cases}
6 , &  k_1,k_2, n+k_3-k_1, n+k_3-k_2 >0,\\
1 , &\kb=\mathbf{0},\\
2 , &\kb=(\frac{n}{3},\frac{n}{3},-\frac{2n}{3}),\\
\frac32 , &\kb=(\frac{n}{2},0,-\frac{n}{2}) \text{ or }  (0,\frac{n}{2},-\frac{n}{2}),\\
3, &otherwise.
\end{cases}
$$
Then $\CL_n$ is the unique function in $\CTS_n$ that satisfies
$
 \CL_nf(\tfrac{\jb}{n}) = f(\tfrac{\jb}{n})$, $\jb \in \Upsilon_n^{\circ}.
$
\end{thm}

\begin{proof}
By the definition of $\CP^{\pm}$ and $\TS_{\kb}$,
\begin{align*}
 \CP^-_{\tb} \Phi(\tb - \tfrac{\jb}{n}) &=  \frac{1}{n^2}\sum_{\kb\in \KK^{\dag*}_n } c_{\wh \kb}^{(3n)}\overline{\phi}_{\kb}(\tfrac{\jb}{n}) \CP^-_{\tb}\phi_{\kb}(\tb)
 =\frac{i}{n^2}\sum_{\kb\in \KK^{\dag*}_n } c_{\wh \kb}^{(3n)}\overline{\phi}_{\kb}(\tfrac{\jb}{n}) \TS_{\kb}(\tb)
\\ & =\frac{i}{n^2} \sum_{\kb\in \Upsilon^{\dag*}_n }\sum_{\sigma\in \A_2}
  c_{\wh {\kb\sigma}}^{(3n)}\overline{\phi}_{\kb\sigma}(\tfrac{\jb}{n})
     \TS_{\kb\sigma}(\tb) \frac{|\kb \A_2|}{|\A_2|}
\\ & =\frac{i}{n^2} \sum_{\kb\in \Upsilon^{\dag*}_n }  c_{\wh {\kb}}^{(3n)}|\kb \A_2| \TS_{\kb}(\tb)  \frac{1}{|\A_2|}\sum_{\sigma\in \A_2}  \sign(\sigma) \overline{\phi}_{\kb\sigma}(\tfrac{\jb}{n})
\\ & =\frac{1}{n^2} \sum_{\kb\in \Upsilon^{\dag*}_n }  c_{\wh {\kb}}^{(3n)}|\kb \A_2| \TS_{\kb}(\tb)  \overline{\TS}_{\kb}(\tfrac{\jb}{n})
= \frac{1}{n^2} \sum_{\kb\in \Upsilon^{\dag*}_n } {\wh \lambda}_{\kb}^{(n)} \TS_{\kb}(\tb)  \overline{\TS}_{\kb}(\tfrac{\jb}{n}).
\end{align*}
Now, for $\jb,\lb\in \Upsilon_n^{\circ}$,
\begin{align*}
 & \CP^-_{\lb}  \Phi(\tb - \tfrac{\jb}{n}) =
 \CP^-_{\lb}\frac{1}{n^2}\sum_{\kb\in \KK^{\dag*}_n } c_{\wh \kb}^{(3n)}
   \phi_{\wh\kb}(\tfrac{\wh\lb}{3n}) \overline{\phi_{\wh\kb}(\tfrac{\wh\jb}{3n})}
= \CP^-_{\lb}\frac{1}{n^2}\sum_{\kb\in \KK^{\dag*}_n } c_{\wh \kb}^{(3n)} \phi_{\tfrac{\wh\lb}{3}}(\tfrac{\wh\kb}{n})
\overline{\phi_{\tfrac{\wh\jb}{3}}(\tfrac{\wh\kb}{n})} \\
& =
\sum_{\sigma\in \A_2}  \frac{\rho(\sigma)}{n^2}\sum_{\kb\in \KK^{\dag*}_n } c_{\wh \kb}^{(3n)} \phi_{\tfrac{\wh{\lb\sigma}}{3}}(\tfrac{\wh\kb}{n})
\overline{\phi_{\tfrac{\wh\jb}{3}}(\tfrac{\wh\kb}{n})}
= \frac16
\sum_{\sigma\in \A_2} \frac{\rho(\sigma)}{n^2}\sum_{\ib\in \KK^{*}_n } c_{\ib}^{(n)} \phi_{\tfrac{\wh{\lb\sigma}}{3}}(\tfrac{\ib}{n})
\overline{\phi_{\tfrac{\wh\jb}{3}}(\tfrac{\ib}{n})}
\\ & = \frac16
\sum_{\sigma\in \A_2}\rho(\sigma) (\phi_{\tfrac{\wh{\lb\sigma}}{3}}, \phi_{\tfrac{\wh\jb}{3}})_n^*
= 
\frac16\sum_{\sigma\in \A_2}\rho(\sigma)\delta^n_{\jb,\lb\sigma} =\frac16 \delta^n_{\jb,\lb},
 \end{align*}
where $\delta^n_{\jb , \kb} $ equals $1$ if $ \frac{\jb}{n}\equiv \frac{\kb}{n} \pmod 3$, and is $0$ otherwise. This completes the proof.
\end{proof}

In fact, $ \ell_{\jb,n}^{\circ} (\tb) = 6 \CP^- \Phi_n(\tb - \tfrac{\jb}{n})$, where $\CP^-$ acts
on the variable $\tb$, from which the proof reduces to verify formula of $\ell_{\jb}^\circ$
given in the theorem, using the periodicity and the symmetry. Applying now $\CP^+$ to $\CI_n f$, we obtain similarly the
trigonometric interpolation on $\Upsilon_n$ in $\Delta$.

\begin{thm}\label{prop:interpo-Dela}
For $n \ge 0$ and $f\in C(\Delta)$ define
\begin{equation*}
  \CL_n^* f(\tb) :=  \sum_{\jb \in \Upsilon_n} f (\tfrac{\jb}{n}) \ell_{\jb,n}(\tb), \quad
     \ell_{\jb,n}(\tb) := \frac{\lambda^{(n)}_{\jb}}{n^2} \sum_{\kb \in \Upsilon_n^{\dag}}
       \lambda_{\kb}^{(n)} \TC_{\kb}(\tb) \overline{ \TC_{\kb}(\tfrac{\jb}{n}) },
\end{equation*}
where $\lambda_\jb^{(n)}$ are defined in \eqref{cuba-HHT2}. Then $\CL_n^*$ is the
unique function in $\CTC_n$ that satisfies
$\CL_n^* f(\tfrac{\jb}{n}) = f(\tfrac{\jb}{n})$, $\jb \in \Upsilon_n.$
\end{thm}

For $n$ being a multiple of $3$, we can deduce a compact formula for  $\ell_{\jb,n}^\circ(\tb)$
and $\ell_{\jb,n}(\tb)$ from that of \eqref{tPhi_n}. The interpolation points of $\CL_n^*f$
are depicted in Figure 8. From the explicit formula of $\Phi_n^*$ in \eqref{tPhi_n}, it is not
difficult to prove, following proof of Theorem 3.6 in \cite{LSX1},  that the uniform operator
norm (Lebesgue constant) of $\CI_{n}^*f$ in Theorem \ref{prop:tH-interpo} satisfies
$\|\CI_{n}^*\|_\infty \le c (\log n)^2$ for $n \equiv 0 \pmod 3$; in other words,
$\|\CI_{n}^* f \|_\infty \le c \|f\|_\infty$,
where $\|\cdot\|_\infty$ denotes the uniform norm over $\Omega$. Since $\CL_n f$ and
$\CL_n^*f$ are obtained by applying $\CP^\pm$ to $\CI_n^*f$, it follows immediately that
$$
   \|\CL_{n} \|_\infty \le c (\log n)^2  \quad{and} \quad  \|\CL_{n}^* \|_\infty \le c (\log n)^2,
$$
where $n \equiv 0 \pmod 3$ and the uniform norm is taken over the triangle $\Delta$.

\subsubsection{Fast Fourier transform} Comparing to the Hexagon-Hexagon case, the
set up in the present subsection has at least one advantage if we consider the fast Fourier transform. The discrete Fourier transform of a function $f$ periodic in $H$ is
$$
 \CI_n f(\tb)  = \sum_{\kb \in \KK_n^\dag} \wh f_\kb \phi_\kb(\tb), \quad\hbox{where}\quad
      \wh f_\kb = \la f, \phi_\kb \ra_n = \frac{1}{n^2}\sum_{\jb\in \KK_n} f(\tfrac{\jb}{n})
         \e^{-\frac{2i\pi}{3} \kb\cdot \jb  }.
$$
For $\la \cdot, \cdot \ra_n$ in \eqref{ipHH2}, we show that $\wh f_\kb$ can be evaluated as
in the classical discrete Fourier transform on a square. For this purpose, it is more
convenient to use Cartesian coordinates. Let $k = (k_1,k_2)$ corresponds to $\kb$.
Then, by Theorem \ref{thm:df},
$$
  \wh f_{\kb}  =   ( f, \phi_{\kb})_n =  \la f, \phi_{k} \ra_N
       =  \frac{1}{n^2} \sum_{j\in \Lambda_{N}} f( n^{-1} H j) \e^{2i\pi n^{-1} k\cdot j},
$$
since $\phi_k(x) = \e^{2 \pi i H^{-1} k \cdot x}$ and $B = n H^{-\tr}$ implies that
$\phi_k(B^{-\tr}j) =  \e^{2i\pi n^{-1} k\cdot j}$. The homogeneous coordinates
of $H j$ is $(2j_1-j_2,2j_2-j_1,-j_1-j_2)$,
so that
\begin{align*}
  \hat f_{\kb}& =\frac{1}{n^2} \sum_{j\in \Lambda_{N}} f( \tfrac{2j_1-j_2}{n},
     \tfrac{2j_2-j_1}{n}, \tfrac{-j_1-j_2}{n}) \e^{2i\pi n^{-1} k\cdot j  } \\
 &=\frac{1}{n^2} \sum_{0\le j_1,j_2<n} f( \tfrac{2j_1-j_2}{n}, \tfrac{2j_2-j_1}{n}, \tfrac{-j_1-j_2}{n})    \e^{2i\pi n^{-1} (k_1 j_1+k_2 j_2)}.
\end{align*}
This states that the discrete Fourier transform coincides, up to a reordering,  with
the classical discrete Fourier transform on a rectangle. Figure \ref{reorderPoint} shows
the set $\Lambda_N$ and its reordering in rectangular coordinates.
\begin{figure}[h]
\centering
\begin{minipage}{0.37\textwidth}\centering \includegraphics[width=1\textwidth]{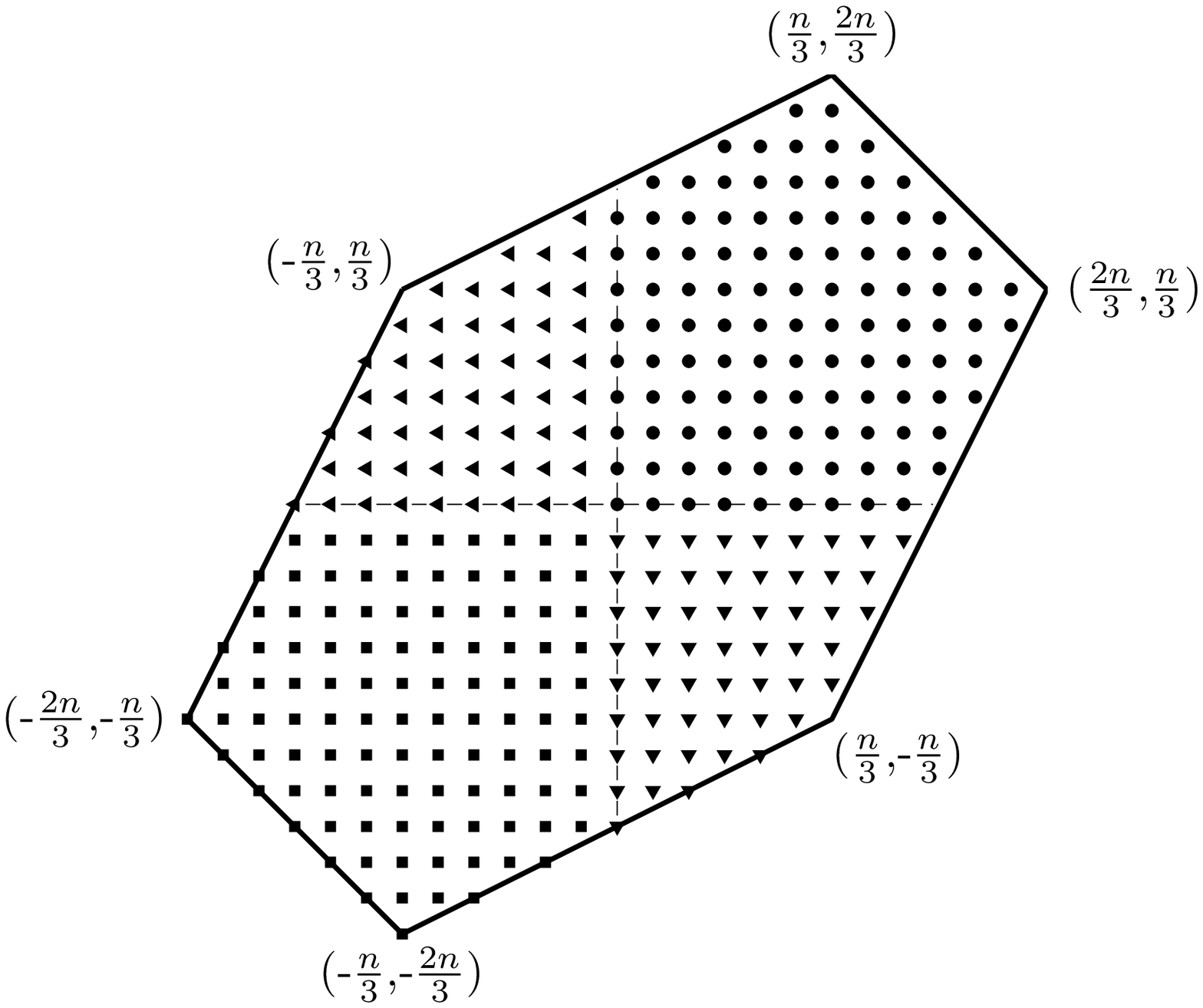}
\end{minipage}
\qquad
\begin{minipage}{0.3\textwidth}\centering \includegraphics[width=1\textwidth]{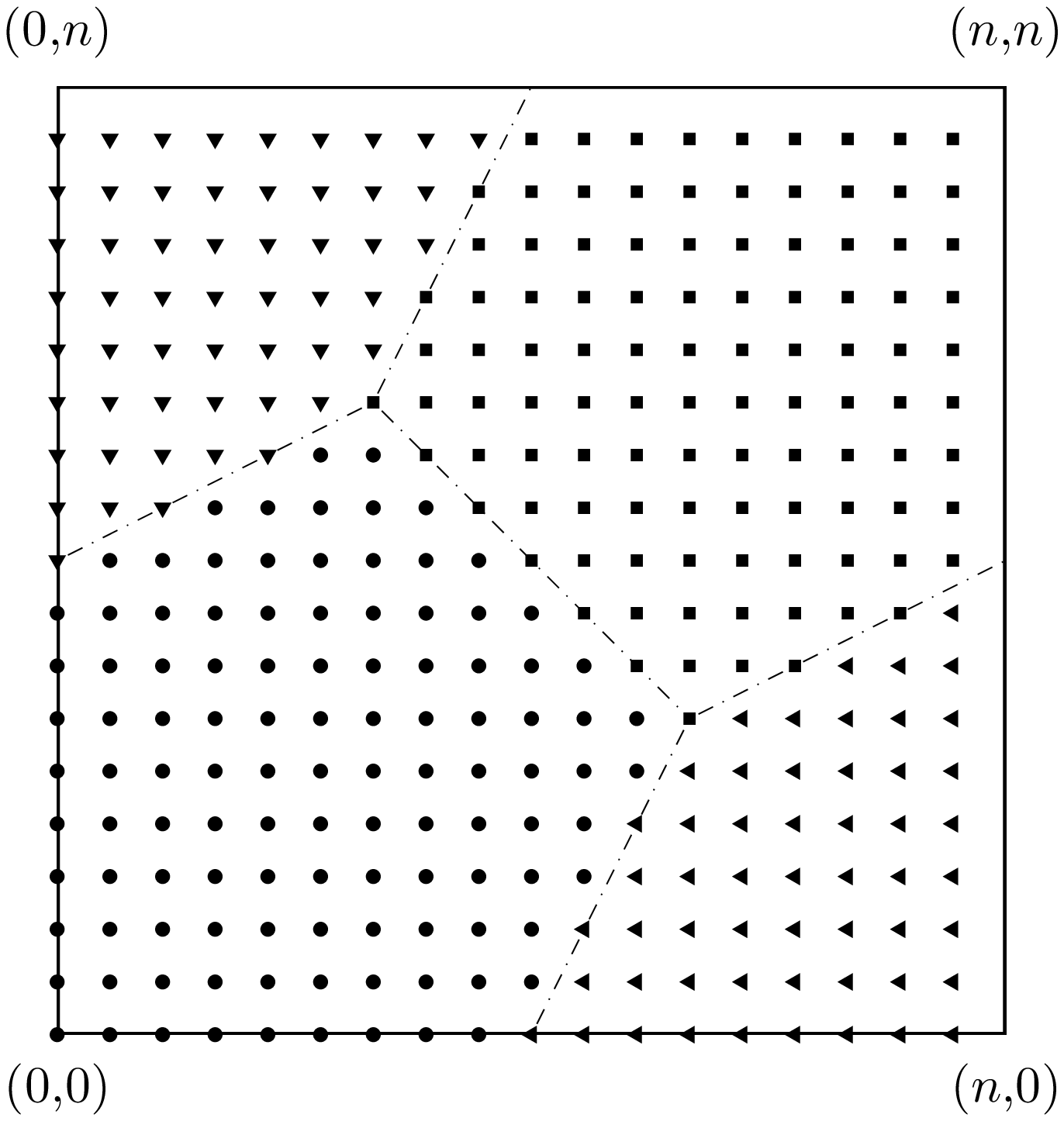}
\end{minipage}
\caption{The index set $\Lambda_N$ (left) and its reordering (right).}
\label{reorderPoint}
\end{figure}
Similarly, recalling $\CH_N =\{\phi_\kb: \kb \in \KK_n^\dag\}$; the index set $\KK_n^\dag$
in rectangular coordinates can also be reordered, so that $\CH_N$ becomes the product
space in rectangular domain. In particular, this allows us to apply the classical FFT to
evaluate $\wh f_\kb$.

\subsection{Other possibilities} There are other possible choices of lattices in our
general frame of discrete Fourier analysis. For example, we can consider
$A = H^{-\tr}$ and $B = n H$, for which the integral domain $\Omega_A$ will be
the hexagon in Figure 5. It is easy to see that the index sets $\Lambda_N$ and
$\Lambda_N^\dag$ in this case are $\KK_n^\dag$ and $\KK_n$ in the previous
subsection, that is, their roles are interchanged. This case, however, does not seem
to lead to interesting new result; the integral domain in the Stage 3 for the generalized
cosine and sine functions will be the quadrilateral in Figure 7.

One obvious question is if we can choose one lattice tiling $\RR^2$ with square
or rhombus and choose the other lattice tiling $\RR^2$ with hexagon. The answer is
negative if we try to use regular hexagon, since the matrix $H$ contains $\sqrt{3}$ and
the requirement $N= B^\tr A$ having all integer entries cannot be satisfied. We can,
however, use other hexagon domains. For example, we can choose either
$$
    H_1 = \left [ \begin{matrix} 1 & 1 \\ -2 & 1 \end{matrix} \right]
       \quad \hbox{or} \quad
   H_2 = \left [ \begin{matrix} 1 & 2 \\ -1 & 1 \end{matrix} \right].
$$
Both lattices $H_1 \ZZ^2$ and $H_2 \ZZ^2$ tile $\RR^2$.
\begin{figure}[h]
\centering
\begin{minipage}{0.3\textwidth}\centering \includegraphics[width=1\textwidth]{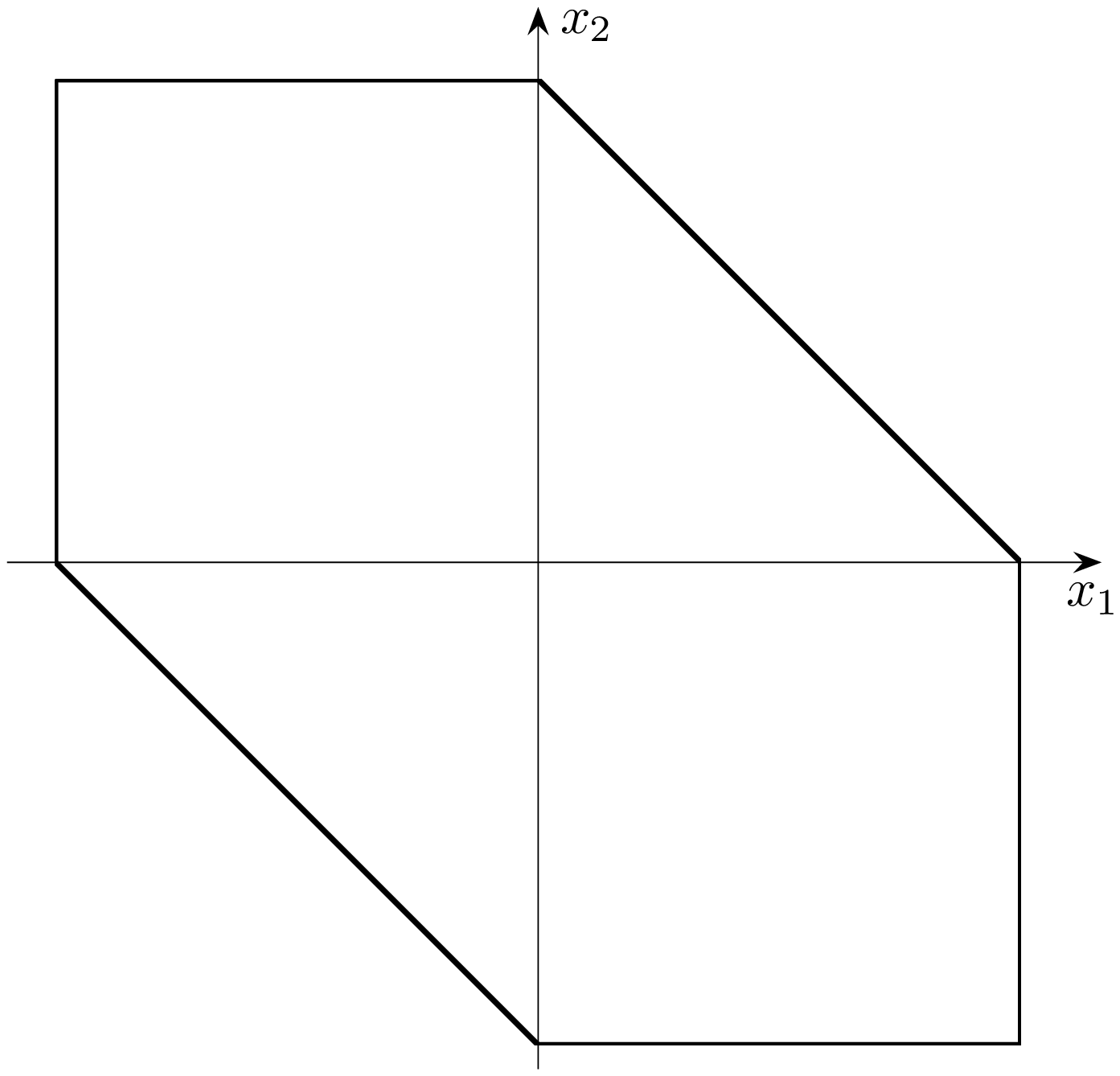}
\end{minipage}
\qquad
\begin{minipage}{0.3\textwidth}\centering \includegraphics[width=1\textwidth]{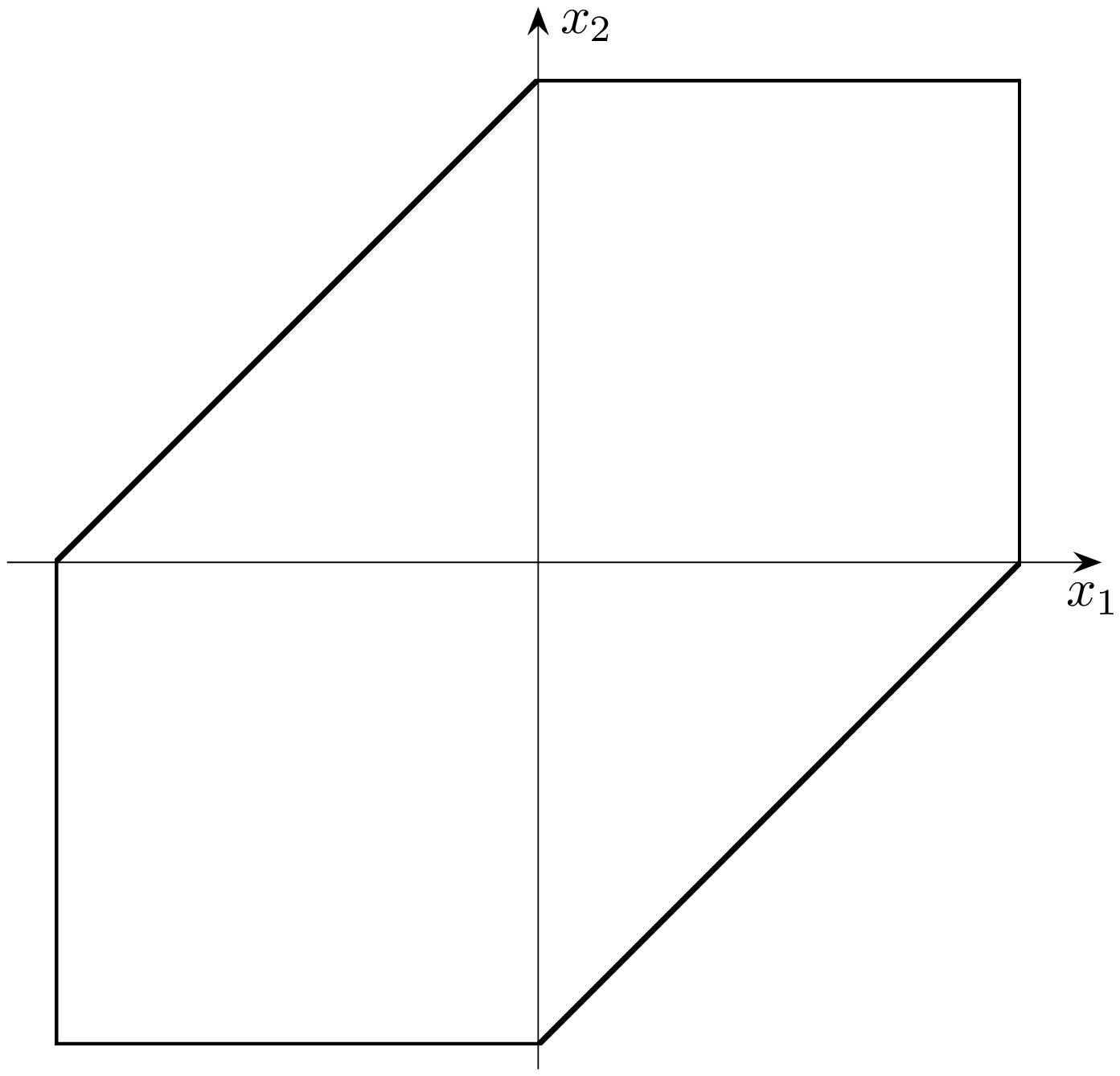}
\end{minipage}
\caption{The fundamental domains of $H_1\ZZ^2$ (left) and $H_2 \ZZ^2$ (right).}
\label{H1-H2}
\end{figure}
Their fundamental domains are depicted in Figure \ref{H1-H2}. The general result in Section 2.1
can be applied to develop a discrete Fourier analysis using either $H_1$ or $H_2$ and
a lattice that tiles $\RR^2$ with either square or rhombus, since the requirement that
$N =B^\tr A$ has integer entries can be readily attained using, say $A= H_1$ or $H_2$
and $B = I$ or $R$. Comparing to the regular
hexagon, the hexagons in Figure 9 possess far less symmetry. The lack of symmetry
means that we will not be able to carry the program outlined in Section 2.2 to Stage 3
and Stage 4, whereas the results in Stage 1 and Stage 2 can be derived from the
general theory straightforwardly. Hence, we will not pursuit the matter any further.

\end{document}